\documentclass[final]{siamart190516}

\usepackage{amsfonts}
\usepackage{graphicx}
\usepackage{epstopdf}
\usepackage{algorithmic}
\ifpdf
  \DeclareGraphicsExtensions{.eps,.pdf,.png,.jpg}
\else
  \DeclareGraphicsExtensions{.eps}
\fi

\newsiamremark{remark}{Remark}
\newsiamremark{hypothesis}{Hypothesis}
\crefname{hypothesis}{Hypothesis}{Hypotheses}
\newsiamthm{claim}{Claim}
\newsiamthm{assumption}{Assumption}

\newsiamremark{example}{Example}

\headers{Long-run risk sensitive impulse control}{D. Jelito, M. Pitera and {\L}. Stettner}

\title{Long-run risk sensitive impulse control}

\author{Damian Jelito\thanks{Institute of Mathematics, Jagiellonian University, Cracow, Poland,
  (\email{damian.jelito@im.uj.edu.pl}, \email{marcin.pitera@im.uj.edu.pl}); Marcin Pitera acknowledges research support by NCN grant 2016/23/B/ST1/00479.}
\and Marcin Pitera\footnotemark[1]
\and  {\L}ukasz Stettner\thanks{Institute of Mathematics, Polish Academy of Sciences, Warsaw, Poland,
  (\email{l.stettner@impan.pl});  research supported by NCN grant 2016/23/B/ST1/00479.}}

\usepackage{amsopn}
\DeclareMathOperator{\E}{\mathbb{E}}

\usepackage{enumitem} 

\newcommand{\vep}{\varepsilon}
\DeclareMathOperator*{\argmin}{arg\,min}

\def\cF{\mathcal{F}}
\def\cT{\mathcal{T}}
\def\bF{\mathbb{F}}
\def\bP{\mathbb{P}}
\def\bE{\mathbb{E}}
\def\bR{\mathbb{R}}
\def\bT{\mathbb{T}}
\def\bN{\mathbb{N}}
\def\bV{\mathbb{V}}

\makeatletter
\def\namedlabel#1#2{\begingroup
    #2%
    \def\@currentlabel{#2}%
    \phantomsection\label{#1}\endgroup
}
\makeatother

\begin{document}

\maketitle

\begin{abstract}
In this paper we consider long-run risk sensitive average cost impulse control applied to a continuous-time Feller-Markov process. Using the probabilistic approach, we show how to get a solution to a suitable continuous-time Bellman equation and link it with the impulse control problem. The optimal strategy for the underlying problem is constructed as a limit of dyadic impulse strategies by exploiting regularity properties of the linked risk sensitive optimal stopping value functions. In particular, this shows that the discretized setting could be used to approximate near optimal strategies for the underlying continuous time control problem, which facilitates the usage of the standard approximation tools. For completeness, we present examples of processes that could be embedded into our framework.
\end{abstract}

\vspace{-0.1cm}

\vspace{-0.1cm}


\section{Introduction}\label{S:introduction}
The risk sensitive control could be considered as a non-linear extension of the classical risk neutral control. While in the standard case the linear expectation operator is applied directly to a reward-cost functional, risk sensitive control is based on exponential utility certainty equivalent with additional risk aversion specification; see~\cite{Whi1990} and references therein. Its links to entropy modeling and applications to finance have been extensively studied in the literature; see~\cite{BiePli1999,Nag2003,PitSte2016}.

While most applied (impulse) controls are in fact based on linear or quadratic objective functionals, the recent rapid growth of computation power made entropy-based controls feasible; see e.g. \cite{ZopParBagSan2020,KusDup2012,BieCheCiaCouJea2019}. In particular, this includes long-run controls, where the potential optimal equilibrium strategies make the framework more robust, invariant to turbulences imposed by frequently unknown termination point, and time-consistent; see~\cite{BiePli1999,BiePli2003,DacLle2014} and references therein.

Due to practical reasons, most applied control systems are based on impulse strategies, as discrete interventions are the only feasible strategies. In consequence, there is a rich literature related to this topic and the impulse control theory is constantly evolving; see e.g. \cite{Rob1981,BenLio1984,DufPiu2016,MenRob2017,PalSte2017}. Still, while impulse control is among the most popular forms of control, its coverage in the risk sensitive case is limited and so far only very specific cases were studied; see~\cite{SadSte2002, Nag2007, PitSte2019, HdiKar2011}. This might be traced back to the fact that the applications of standard methods from the risk neutral case often lead to difficult problems involving quasi variational inequalities and differential equations techniques; see e.g.~\cite{Nag2006,DavGuo2010,BelChr2017,AraBis2020}. Consequently, development of novel methods and proof techniques is required.

In this paper we follow a probabilistic approach and extend results from~\cite{Rob1981}, where the risk neutral optimal stopping problem is linked to the impulse control; see also \cite{PesShi2006,Nag2007b,MenRob2017}. Namely, we consider the long-run risk sensitive impulse control problem
\begin{equation}\label{eq:intro2}
\limsup_{T\to\infty} \frac{1}{T} \ln \E_{(x,V)} \left[\exp\left({\int_0^T f(X_s) ds + \sum_{i=1}^{\infty} 1_{\{\tau_i\leq T\}}  c(X_{\tau_i^-},\xi_i)}\right)\right],
\end{equation}
where the Feller-Markov process $X=(X_t)$ starting in $x$ is controlled by impulse strategies $V$ with a {\it reward} function $f$ and a {\it shift-cost} function $c$. The impulse strategies $V=(\tau_i,\xi_i)_{i\in\bN}$ are described by sequences $(\tau_i)$ and $(\xi_i)$  indicating impulse times and after-impulse states of the process, respectively, and $X_{\tau_i^-}$ denotes the state of $X$ right before the impulse.
Hence, up to time $\tau_1$ the process $X$ evolves according to its usual dynamics, then it is shifted to $\xi_1$ and starts its evolution again. Moreover, we assume that $f$ and $c$ are continuous and bounded, while shifts take values in a compact set. Expectation $\bE_{(x,V)}$ is defined on a probability space corresponding to the controlled process; see Section~\ref{S:preliminaries} for details. We refer to \cite{PalSte2017,PitSte2019,SadSte2002} where similar framework has been studied.

The main result of this paper states that under suitable assumptions there exists an optimal control to \eqref{eq:intro2}. It is also shown that (near-optimal) strategies could be obtained using standard approximation techniques based on a simplified dyadic setting. Those novel results prove that the risk-sensitive control could be efficiently embedded into the impulse control framework and feasible strategies could be constructed. We want to emphasize that so far the problem of the generic form~\eqref{eq:intro2} was studied only in~\cite{SadSte2002}, where a very specific risk-sensitive setting was considered; see also~\cite{Nag2007, PitSte2019, HdiKar2011}.

In our approach, the optimal strategy is constructed as a limit of dyadic impulse strategies by exploiting results from~\cite{PitSte2019} and~\cite{JelPitSte2019a}, i.e. by combining discrete time existence results that are based on the span-contraction approach with regularity properties of risk sensitive optimal stopping value functions. More explicitly, using the change of measure technique based on the Multiplicative Poisson Equation (MPE) solution, we establish a link between impulse control Bellman equation and the optimal stopping problem considered in~\cite{JelPitSte2019a}. As expected, the strategy is characterized by the relation between optimal values of the dyadic problems and the type of the semigroup associated with the underlying Markov process. To get the characterization result, the finite time horizon equivalent of \eqref{eq:intro2} could be rewritten as a limit of problems in which dyadic strategies with a finite number of impulses are considered.

While the main focus of this paper is put on theoretical (existential) aspects of the problem, the results might be of interest to practitioners. In a nutshell, our paper provides a link between discrete (dyadic) setting and continuous-time setting by laying out conditions under which one could use discretized approximative algorithms to efficiently solve the continuous problem. Indeed, we show that a series of $\varepsilon$-optimal strategies for the correctly specified dyadic problem could be used to obtain a nearly optimal strategy for the original problem; we discuss this point in details in Section~\ref{S:continuous_control}, after stating Theorem~\ref{th:opt_strategy_continuous}. We refer to \cite{KusDup2012} and references therein for the description of the approximation techniques used in the stochastic control theory; see also \cite{MihNeu2002,ZopParBagSan2020} for methods based on machine learning that include reinforcement learning and neural networks in the context of (risk sensitive) stochastic control.

This paper is organized as follows:  In Section~\ref{S:preliminaries}, we establish the framework and the linked notation that is used throughout the paper. Section~\ref{S:discrete_control} discusses the dyadic case; this section incorporates the results from \cite{PitSte2019}. Next, in Section~\ref{S:continuous_control}, we solve problem~\eqref{eq:intro2}; Theorem~\ref{th:existence_Bellman_continuous} and Theorem~\ref{th:opt_strategy_continuous} might be seen as the main contribution of this paper. Section~\ref{S:examples} presents three examples of processes that could be embedded into our framework. For reader convenience, we moved auxiliary results to the appendix. In Appendix~\ref{S:stopping}, we outline some properties of the risk-sensitive optimal stopping framework that are used in the proofs; this is an extract from~\cite{JelPitSte2019a}. In Appendix~\ref{S:finite_proof}, we present supplementary results linked to the finite time impulse control problem.

\section{Preliminaries}\label{S:preliminaries}
Let  $(\Omega,\cF,\bF,\bP)$ be a continuous time filtered probability space. We assume that $\bF:=(\cF_t)_{t\in \bT}$, $\bT:=\bR_+$, $\cF_0$ is trivial, $\cF\supseteq \bigcup_{t\in\bT}\cF_t$, and the usual conditions are satisfied. We use $X:=(X_t)$ to denote a continuous-time standard Markov process taking values in a locally compact separable metric space $E$ endowed with metric $\rho$ and Borel $\sigma$-field $\mathcal{E}$; see Definition 4 in \cite[Section 1.4]{Shi1978} for details. Throughout the paper we fix $\overline{x}\in E$ and define $\Vert x\Vert:=\rho(x,\overline{x})$ for $x\in E$; note that $\|\cdot\|$ might not be a norm. Moreover, with slight abuse of notation, for any $f\colon E\to \bR$ we use $\|f\|$ to denote the supremum norm of $f$. We use $C(E)$ to denote the family of continuous and bounded functions on $E$ and assume that $X$ satisfies the $C$-Feller property, i.e. for any $t>0$ we have $P_tC(E)\subset C(E)$, where $(P_t)$ is the transition semigroup of $X$.

As in \eqref{eq:intro2}, with slight abuse of notation, we also use $X$ to denote a controlled process. For any $V=(\tau_i,\xi_i)_{i\in\bN}$, we assume that the sequence of stopping times $(\tau_i)$ is increasing while the shifts $(\xi_i)$ are adapted to the filtration and take values in a compact set $U\subseteq E$, i.e. that $\xi_i$ is $\cF_{\tau_i}$-measurable and $\xi_i \in U$, for any $i\in\bN$. For brevity, we use $\bV$ to denote the set all (admissible) control strategies. For any $x\in E$ and $V\in\bV$ we consider a controlled process probability space that is constructed following the logic from \cite{Rob1978}. In a nutshell, we consider a countable product of canonical spaces of {\it c\`{a}dl\`{a}g} functions with values in $E$, and with the inductively defined filtration. For $t\in [\tau_{i-1},\tau_{i})$ we set $X_t=x_t^{i-1}$, where $x_t^{i-1}$ corresponds to $(i-1)$th coordinate of the canonical process; see also \cite[Section 2]{Ste1983} or \cite{PalSte2017} for more details. For clarity, we use $\bP_{(x,V)}$ and  $\bE_{(x,V)}$ to denote the probability measure and expectation operator corresponding to particular choice of $x\in E$ and $V\in\bV$. For brevity, we use $\bE_{x}$ and $\bP_x$ in reference to the uncontrolled process dynamics.

Our aim is to maximize the long-run version of the risk sensitive criterion with negative (risk averse) parameter $\gamma<0$; see \cite{PitSte2019} for details. After risk aversion parameter standardization, our objective is to minimize the functional
\begin{equation}\label{eq:RSC_problem_min}
J(x,V):=\limsup_{T\to\infty} \frac{1}{T} \ln \E_{(x,V)} \left[\exp\left({\int_0^T f(X_s) ds + \sum_{i=1}^{\infty} 1_{\{\tau_i\leq T\}}  c(X_{\tau_i^-},\xi_i)}\right)\right],
\end{equation}
where $x\in E$ is the starting point, $f\colon E\to \bR$ and $c:\bE\times U\to \bR$ are the (normalized) reward and shift-cost functions, $V\in\bV$ is the control strategy, and $X_{\tau_i^-}$ is the state of the process right before the $i$th impulse, i.e. $X_{\tau_i^-}=x_{\tau_i}^{i-1}$.

Let $\mathcal{T}$ denote the family of almost surely finite stopping times taking values in $\bT$, and let $\mathcal{T}_{m}\subset \cT$ denote the dyadic stopping times defined on the time-grid $\{0,\delta_m, 2\delta_m, \ldots\}$, where $\delta_m:=(1/2)^m$ and $m\in\bN$. By $\bV_m\subset \bV$ we denote the family of impulse control strategies with stopping times restricted to $\mathcal{T}_m$.

In order to solve~\eqref{eq:RSC_problem_min}, we will show the existence of a function $w\in C(E)$ and a constant $\lambda\in \bR$ that satisfy Bellman equation
\begin{equation}\label{eq:opt_stop}
e^{w(x)}=\inf_{\tau}\E_x\left[e^{\int_0^{\tau} (f(X_s)-\lambda)ds+Mw(X_{\tau})}\right], \quad x\in E,
\end{equation}
where the operator $M\colon C(E)\to C(E)$ is defined as
\begin{equation}\label{eq:def_M}
M w(x):=\inf_{\xi \in U}\left(c(x,\xi)+w(\xi)\right).
\end{equation}
Of course, one needs additional assumptions imposed on $X$ and related functions in order to have a proper non-degenerate solution to~\eqref{eq:opt_stop}. Our approach is based on the dyadic approximations for the linked (discrete) stopping time problem; see \cite{PitSte2019}.

\newpage

The set of assumptions that are used throughout the paper is partly based on those introduced in \cite{PitSte2019} and \cite{SadSte2002}:

\medskip

\begin{enumerate}
\item[(\namedlabel{A1}{$\mathcal{A}$1})] (Reward function constraints.) The function $f:E\mapsto \mathbb{R}$ is continuous and bounded.
\item[(\namedlabel{A2}{$\mathcal{A}$2})] (Cost function constraints.) The function $c:E\times U\mapsto \mathbb{R}$ is continuous and bounded. Moreover, $c$ is bounded away from zero by $c_0>0$ and satisfies the triangle inequality, i.e.
\begin{equation}\label{eq:c_multiple_prevent}
c_0 \leq c(x,y)\leq c(x,z)+c(z,y),\quad x\in E,\, y,z\in U.
\end{equation}
Also, $c$ satisfies {\it uniform limit at infinity} condition, i.e.
\begin{equation}\label{assumpt:cost_function}
\lim_{\Vert x\Vert,\Vert y\Vert\to \infty}\sup_{\xi \in U}|c(x,\xi)-c(y,\xi)|=0.
\end{equation}

\item[(\namedlabel{A3}{$\mathcal{A}$3})] (Distance control.) For any compact set $\Gamma\subset E$, $t_0>0$, and $r_0>0$ we have
\begin{equation}\label{Co}
\lim_{r\to\infty}M_{\Gamma}(t_0,r)=0,\qquad \lim_{t\to 0} M_{\Gamma}(t,r_0) =0,
\end{equation}
where $M_{\Gamma}(t,r):= \sup_{x\in \Gamma} \mathbb{P}_x[\sup_{s\in [0,t]} \rho(X_s,x)\geq r]$.

\item[(\namedlabel{A4}{$\mathcal{A}$4})] (Mixing condition.) For any $t>0$ we have $\mathcal{E}_t(x)=\mathcal{E}_t(y)$, $x,y\in E$, and 
\[
\sup_{x,y\in E}\sup_{A \in\mathcal{E}_t(y)} \frac{\mathbb{P}_x[X_t\in A]}{\mathbb{P}_y[X_t\in A]} <\infty,
\]
where $\mathcal{E}_t(y):=\{A \in\mathcal{E}\colon \mathbb{P}_y[X_t\in A]>0\}$. Also, $U \in \mathcal{E}_t(x)$, for $t>0$ and $x\in E$.
\end{enumerate}

\medskip

Assumption~\eqref{A1} is a classic reward function constraint and directly reflects assumption (A1) in \cite{PitSte2019} for bounded weight function.

Assumption~\eqref{A2} imposes multiple restrictions on the cost function and partly reflects assumption (A2) in \cite{PitSte2019}. For completeness, let us provide exemplary class of functions which satisfy~\eqref{A2}. Let $h\colon \bR_+\to \bR_+$ be bounded, continuous, non-decreasing,
and satisfy subadditivity condition $h(x+y)\leq h(x)+h(y)$, $x,y\geq 0$. Then, the function $c(x,\xi):=h(\rho(x,\xi))+c_0$ satisfies~\eqref{A2}. For example, we can set $h_1(x):=x\wedge K$ (where $K\geq 0$ is fixed), $h_2(x)=\frac{x}{1+x}$, or $h_3(x):=\frac{1}{1+e^{-x}}$.

Assumption~\eqref{A3} facilitates distance control of the uncontrolled process. The property $\lim_{r\to\infty}M_{\Gamma}(t_0,r)=0$ implies that for any fixed time-point $t_0>0$ the process cannot get too far from the starting point, while property $\lim_{t\to 0}M_{\Gamma}(t,r_0)=0$ indicates that the process does not exhibit frequent sudden jumps on small time intervals. In particular, both properties are satisfied if the underlying space is compact or the process satisfies the $C_0$-Feller property; see \cite{JelPitSte2019a} and \cite{BasSte2018} for details.

Assumption~\eqref{A4} quantifies the ergodic properties of the underlying uncontrolled process $X$. Equivalently, we can state that for any $t>0$ there exists a probability measure $\nu_t:\mathcal{E}\to [0,1]$, density $p_t: E\times E\to \bR_{+}$ and constants $0<a_t\leq b_t<\infty$,  such that 
\begin{equation}\label{eq:A4.alt}
\bP_x(X_t\in A)=\int_A p_t(x,y)\nu_t(dy), \quad A\in \mathcal{E},
\end{equation}
with $\nu_t(U)>0$ and $a_t\leq p_t(x,y)\leq b_t$, for $x,y\in E$. In particular, this condition is satisfied for regular reflected diffusions; see Section~\ref{S:examples} for details. Also, it should be noted that~\eqref{A4} implies the global minorization property as well as the existence of a solution to the Multiplicative Poisson Equation (MPE) linked to~\eqref{eq:opt_stop}. 
\newpage
\noindent Namely, under~\eqref{A4}, we get the following two properties:

\medskip

\begin{enumerate}
\item[(\namedlabel{A4a}{$\mathcal{A}$4a})] (Global minorization.) For any $t>0$ there exists $a_t>0$ and a probability measure $\nu_t,$ such that $\nu_t(U)>0$ and $\inf_{x\in E} \mathbb{P}_x(X_{t}\in A)\geq a_t \nu_t(A)$, for $A\in \mathcal{E}$.
\item[(\namedlabel{A4b}{$\mathcal{A}$4b})]  (Existence of MPE solution.) There exists $v\in C(E)$ satisfying equation $v(x)=\ln\E_x[\exp(\int_0^t (f(X_s)-r(f)) ds +v(X_{t}))]$, for all $t\geq 0$ and $x\in E$, where $r(f):=\inf_{t>0}\frac{1}{t} \Vert H_t\Vert$ denotes the semi-group type, $\Vert \cdot \Vert$ is the supremum norm, and $H_t(x):=\ln\E_x[\exp(\int_0^t f(X_s) ds)]$.
\end{enumerate}

\medskip

\noindent We refer to \cite[Theorem 2.6]{DiMSte1999}, \cite[Corollary 2]{Ste1989}, and \cite[Lemma 3.2]{SadSte2002} for details. In fact, all the proofs presented in this paper are valid if we only assume global minorization and existence of MPE solution, i.e. replace~\eqref{A4}  with \eqref{A4a} and \eqref{A4b}. Nevertheless, we introduced stronger condition \eqref{A4} to simplify the narrative.

Before we proceed, let us provide a comment on the implications of assumptions \eqref{A1}--\eqref{A4}. First, note that assumption \eqref{A4} implies (A4) from \cite{PitSte2019} since (A4) becomes global minorization property in the bounded framework. Thus, noting that assumption (A3) from \cite{PitSte2019} is trivially satisfied due to boundedness of $f$ and $c$, we get that our setting is consistent with assumptions (A1)--(A4) introduced in \cite{PitSte2019}. Second,~\eqref{A4} allows us to apply change of measure technique that simplifies problem~\eqref{eq:opt_stop} by reducing it to the optimal stopping setting considered in \cite{JelPitSte2019a}. Indeed, let~\eqref{A4b} hold and, for $x\in E$, let $\mathbb{Q}_x$ denote a probability measure given via a Radon-Nikodym derivative
\begin{equation}\label{eq:measure_Q}
d \mathbb{Q}_x\big|_{t} := Y_t(x) d \mathbb{P}_x\big|_{t},\qquad t\geq 0,
\end{equation}
where $Y_t(x):=e^{-v(x)}e^{\int_0^t (f(X_s)-r(f)) ds +v(X_{t})}$, $t\geq 0$; note that $Y_t(x)$ is a martingale and $\E_x\left[ Y_t(x)\right]=1$, $x\in E$. Then, based on the identity
\begin{align}\label{rm:changed_measure}
\E_x  \left[e^{\int_0^{\tau} (f(X_s)-\lambda)ds+G(X_{\tau})}\right]& =e^{v(x)} \E_x^{\mathbb{Q}}\left[e^{(r(f)-\lambda)\tau+G(X_{\tau})-v(X_\tau)}\right],
\end{align}
where $x \in E$, $\lambda\in \bR$, $G\in C(E)$, $\tau\in \mathcal{T}$, and $\E_x^{\mathbb{Q}}$ denotes expectation with respect to $\mathbb{Q}_x$, we are able to reduce $(f(\cdot)-\lambda)$ to a constant in~\eqref{eq:opt_stop}. Namely, taking into account \eqref{rm:changed_measure}, Bellman equation~\eqref{eq:opt_stop} could be rewritten as
\begin{equation}\label{eq:optstop_Q}
e^{w(x)-v(x)}=\inf_{\tau}\E_x^{\mathbb{Q}}\left[e^{(r(f)-\lambda)\tau+Mw(X_{\tau})-v(X_\tau)}\right], \quad x\in E.
\end{equation}
Now, if $\lambda<r(f)$, then \eqref{eq:optstop_Q} falls within the class of optimal stopping problems that were recently studied in~\cite{JelPitSte2019a}; see Appendix~\ref{S:stopping} for details. The remaining degenerate case $\lambda=r(f)$ (we will show that $\lambda\leq r(f)$) requires special tools and corresponds to the situation when the \textit{no impulse} strategy is optimal; see Theorem~\ref{th:opt_strategy_continuous} for details.

\section{Dyadic impulse control}\label{S:discrete_control}
In this section we outline the properties of the dyadic version of the optimal impulse control problem~\eqref{eq:RSC_problem_min}. Recall that for any $m\in\bN$ and time-step $\delta_m=\tfrac{1}{2^m}$, the corresponding family of dyadic stopping times taking values in $\{0,\delta_m, 2\delta_m, \ldots\}$ is denoted by $\mathcal{T}_{m}$. For any $m\in\bN$, the dyadic version of Bellman equation~\eqref{eq:opt_stop} is given by
\begin{equation}\label{eq:opt_stop_discrete}
e^{w_{m}(x)}=\inf_{\tau\in \mathcal{T}_{m}}\E_x\left[e^{\int_0^{\tau} (f(X_s)-\lambda_{m})ds+Mw_{m}(x_{\tau})}\right],\quad  x\in E,
\end{equation}
where $M$ is defined in \eqref{eq:def_M}, $w_m\in C(E)$ and $\lambda_m\in\bR$. In fact, due to the dyadic nature of the problem, one could consider the associated one-step equation given by
\begin{equation}\label{eq:fixed_point_Bellman}
e^{w_m(x)}=\min\left(\E_x\left[e^{\int_0^{\delta_m}(f(X_s)-\lambda_m)ds+w_m(X_{\delta_m})}\right],e^{Mw_m(x)}\right), \quad x\in E.
\end{equation}
Under suitable assumptions, solution to~\eqref{eq:fixed_point_Bellman} exists and solves the dyadic version of the optimal control problem, i.e.
\begin{equation}\label{eq:dyadic:opt}
\inf_{V\in\bV_m}J(x,V),
\end{equation}
where $\bV_m$ is the set of admissible control strategies on the time-grid spanned by $\delta_m$.
\begin{theorem}\label{th:PS_existence}
Under Assumptions~\eqref{A1}--\eqref{A4}, for any $m\in \bN$, there exists a unique (up to an additive constant) function $w_m\in C(E)$ and a constant $\lambda_m\in \mathbb{R}$ satisfying~\eqref{eq:fixed_point_Bellman}. Moreover, $\lambda_m$ is the optimal value of the problem~\eqref{eq:dyadic:opt}.
\end{theorem}

\begin{proof}
The proof relies on the results presented in~\cite{PitSte2019}; it should be noted that while in~\cite[Equation 2.5]{PitSte2019} the maximization problem is considered, it directly corresponds to standardized problem~\eqref{eq:dyadic:opt} due to negative sign of the risk aversion parameter in the entropic utility measure. Also, note that in our setting all assumptions listed in~\cite[Section 2]{PitSte2019} are satisfied. Please recall that~\cite[Assumption A.3]{PitSte2019} is naturally satisfied in the bounded framework, while other assumptions have a direct link to assumptions stated in this paper; see Section~\ref{S:preliminaries} for details.

First, from \cite[Proposition 3.4]{PitSte2019} we get that there exists a unique (up to an additive constant) function $w_m\in C(E)$ and a constant $\lambda_m\in \mathbb{R}$ satisfying equation
\begin{multline}\label{eq:th:PS:fixed_point}
e^{w_m(x)}=\min\left(\E_x\left[e^{\int_0^{\delta_m}(f(X_s)-\lambda_m)ds+w_m(X_{\delta_m})}\right]\right.,\\
\left.\inf_{\xi \in U} e^{c(x,\xi)} \E_{\xi}\left[e^{\int_0^{\delta_m}(f(X_s)-\lambda_m)ds+w_m(X_{\delta_m})}\right]\right).
\end{multline}
Second, let us show that
\begin{equation}\label{eq:dyadic1}
\inf_{\xi \in U} e^{c(x,\xi)} \E_{\xi}\left[e^{\int_0^{\delta_m}(f(X_s)-\lambda_m)ds+w_m(X_{\delta_m})}\right]=e^{Mw_m(x)}.
\end{equation}
For brevity, we define
\[
Z(x):=\E_x\left[e^{\int_0^{\delta_m}(f(X_s)-\lambda_m)ds+w_m(X_{\delta_m})}\right], \quad x\in E.
\]
From~\eqref{eq:th:PS:fixed_point}, we get
$
Z(x) \geq e^{w_m(x)}
$
for any $x\in E.$ Therefore, for any $x\in E$, we get
\begin{equation}\label{cor:fixed_point_simple:eq1}
\inf_{\xi \in U} e^{c(x,\xi)} Z(\xi)\geq\inf_{\xi \in U} e^{c(x,\xi)} e^{w_m(\xi)}.
\end{equation}
On the other hand, due to~\eqref{eq:c_multiple_prevent} and~\eqref{eq:th:PS:fixed_point}, we get
\begin{align*}
\inf_{\xi \in U} e^{c(x,\xi)} Z(\xi) &= \inf_{\xi \in U} e^{c(x,\xi)}Z(\xi)\wedge \inf_{\zeta \in U} e^{c(x,\zeta)} Z(\zeta)\\
& \leq \inf_{\xi \in U} \left(e^{c(x,\xi)}Z(\xi)\wedge \inf_{\zeta \in U} e^{c(x,\xi)} e^{c(\xi,\zeta)} Z(\zeta)\right)\\
& =\inf_{\xi \in U} e^{c(x,\xi)} \left(Z(\xi)\wedge \inf_{\zeta \in U} e^{c(\xi,\zeta)} Z(\zeta)\right)\\
& = \inf_{\xi \in U} e^{c(x,\xi)} e^{w_m(\xi)},
\end{align*}
which proves~\eqref{eq:dyadic1}. Combining~\eqref{eq:th:PS:fixed_point} and~\eqref{eq:dyadic1} we get that $\lambda_m$ and $w_m$ are solutions to~\eqref{eq:fixed_point_Bellman}.

Finally, using \cite[Proposition 4.3]{PitSte2019} we can link the constant $\lambda_m$ from~\eqref{eq:fixed_point_Bellman} with the optimal value of the dyadic control problem~\eqref{eq:dyadic:opt}, which concludes the proof.
\end{proof}

Now, we link the solution of~\eqref{eq:fixed_point_Bellman} to \eqref{eq:opt_stop_discrete}. The link is related to the type of the semigroup
\[
r(f)=\inf_{t>0}\frac{1}{t} \Vert H_t\Vert,
\]
where $H_t(x)=\ln\E_x[\exp(\int_0^t f(X_s) ds)]$, $x\in E$; cf. Property \eqref{A4b}.
\begin{proposition}\label{pr:r(f)_cases}
For any $m\in \mathbb{N}$ we get $\lambda_m\leq r(f)$. Moreover,
\begin{enumerate}
\item if $\lambda_m=r(f)$, then the {\it no impulse} strategy is optimal for $\mathbb{V}_m$;
\item if $\lambda_m<r(f)$, then $w_m$ defined in~\eqref{eq:fixed_point_Bellman} satisfies~\eqref{eq:opt_stop_discrete}.
\end{enumerate}
\end{proposition}

\begin{proof}
Let us fix $m\in\bN$. First, we show that for any $x\in E$, we get $\lambda_m\leq r(f)$. Using \cite[Proposition 1]{Ste1989}, we get that $r(f)$ may be rewritten as $r(f)=\lim_{T\to\infty}\tfrac{1}{T} \| H_T\|$. Thus, for any $x\in E$ and the {\it no impulse} strategy $V_0\in\bV$, we get
\[
J(x,V_0)=\limsup_{T\to\infty}\frac{1}{T}\ln  H_T(x)\leq \limsup_{T\to\infty}\frac{1}{T}\sup_{x\in E} H_T(x)= r(f).
\]
Also, using Theorem~\ref{th:PS_existence}, we know that for $\lambda_m$ we get
\begin{equation}\label{eq:r(g)_bound}
\lambda_m=\inf_{V\in\bV_m}J(x,V) \leq J(x, V_0)\leq r(f),
\end{equation}
which concludes the proof of $\lambda_m\leq r(f)$.

Now, let $\lambda_m=r(f)$. From~\eqref{eq:r(g)_bound} we get $r(f)=J(x,V_0)$. This implies optimality of the {\it no impulse} strategy.

Next, let $\lambda_m<r(f)$. To show that $w_m$ defined in~\eqref{eq:fixed_point_Bellman} satisfies~\eqref{eq:opt_stop_discrete} we use the change of measure technique based on MPE. In this way we can replace the term $(f(\cdot)-\lambda_m)$ in~\eqref{eq:fixed_point_Bellman} by some positive constant and use the results from~\cite[Section 3]{JelPitSte2019a}; see also Appendix~\ref{S:stopping}. Recalling $v$ from~\eqref{A4b}, $w_m$ from~\eqref{eq:fixed_point_Bellman}, and $\mathbb{Q}_x$ from~\eqref{eq:measure_Q}, we set
\begin{align}\label{eq:def_aux1}
u_m(x)& :=w_m(x)-v(x)+\Vert w_m \Vert +\Vert v\Vert+\Vert Mw_m \Vert,\nonumber\\
G_m(x)& :=Mw_m(x)-v(x)+\Vert w_m \Vert +\Vert v\Vert+\Vert Mw_m \Vert,\nonumber\\
S^m h(x)& :=\E_x^{\mathbb{Q}}\left[e^{(r(f)-\lambda_m)\delta_m} h(X_{\delta_m})\right]\wedge e^{G_m(x)}, \quad h \in C(E).
\end{align}
Using~\eqref{rm:changed_measure} and~\eqref{eq:fixed_point_Bellman} we get $S^m e^{u_m} = e^{u_m}$. Also, noting that $G_m\in C(E), G_m(\cdot)\geq 0,$ $r(f)-\lambda_m>0$, and using Proposition~\ref{pr:stopping_Bellman}, we get
\[
e^{u_m(x)}=\inf_{\tau\in \mathcal{T}_m}\E_x^{\mathbb{Q}}\left[e^{\tau(r(f)-\lambda_m)+G_m(X_\tau)}\right].
\]
Recalling~\eqref{rm:changed_measure} and~\eqref{eq:def_aux1} we conclude that $w_m$ satisfies~\eqref{eq:opt_stop_discrete}.
\end{proof}



\section{Continuous time impulse control}\label{S:continuous_control}

In this section we study the impulse control problem~\eqref{eq:RSC_problem_min} and linked Bellman equation~\eqref{eq:opt_stop}. Our approach is based on the approximation of Bellman equation by its dyadic version introduced in Section~\ref{S:discrete_control}. 

Let $(\lambda_m)$, $m\in\bN$, be a sequence of solutions to dyadic Bellman equations \eqref{eq:opt_stop_discrete}; from Theorem~\ref{th:PS_existence} we know that this sequence exists. Since $(\lambda_m)$ is decreasing and  $\lambda_m\geq -\Vert f\Vert$, we can define the finite limit
\begin{equation}\label{eq:lambda}
\lambda:=\lim_{m\to\infty} \lambda_m.
\end{equation}
From Proposition~\ref{pr:r(f)_cases} we know that $\lambda\leq r(f)$. In Theorem~\ref{th:existence_Bellman_continuous}, we show that if $\lambda<r(f)$, then there exists a solution to~\eqref{eq:opt_stop}.
\begin{theorem}\label{th:existence_Bellman_continuous}
Assume~\eqref{A1}--\eqref{A4} and let $\lambda$ be given by~\eqref{eq:lambda}. If $\lambda<r(f)$, then there exists a function $w\in C(E)$, such that $w$ and $\lambda$ are solutions to~\eqref{eq:opt_stop}.
\end{theorem}

\begin{proof}
For transparency, we split the proof into two steps: (1) proof of the fact that $(Mw_{m_n})\to \phi$ uniformly to some $\phi\in C(E)$, along a subsequence $(m_n)$; (2) proof of the identity $\phi=Mw$ for a suitable $w$. In the end, we comment how to combine those steps to get \eqref{eq:opt_stop} and conclude the proof.

\noindent {\it Step 1.} Let $(w_m)$ be the sequence of functions $w_m\in C(E)$ given by
\begin{equation}\label{eq:def.w.m}
w_m(x):=\widetilde{w}_m(x)-\inf_{\xi\in U}\widetilde{w}_m(\xi),
\end{equation}
where $\widetilde{w}_m$, $m\in\bN$, is a solution to the Bellman equation~\eqref{eq:fixed_point_Bellman}; note that $\widetilde{w}_m$ exists due to Theorem~\ref{th:PS_existence} and $(w_m)$ also satisfies~\eqref{eq:fixed_point_Bellman}. Now, we show that using Arzel\`a-Ascoli Theorem one can choose uniformly convergent subsequence of $(Mw_m)$, where $M$ is given in~\eqref{eq:def_M}.

First, we show that $(Mw_m)$ is uniformly bounded on $E$. For any $m\in\bN$ and $\xi\in U$, we get $w_m(\xi)\geq 0$. Consequently, for $x\in E$, we get $Mw_m(x)\geq 0$, i.e. we found the uniform lower bound for $(Mw_m)$. On the other hand, recalling \eqref{eq:fixed_point_Bellman} and \eqref{eq:def.w.m}, for any $x\in E$ and $m\in \bN$, we get
\[
w_{m}(x)\leq M \widetilde w_{m}(x) -\inf_{\xi\in U}\widetilde{w}_m(\xi)\leq \left(\Vert c\Vert+\inf_{\xi\in U}\widetilde{w}_m(\xi)\right)-\inf_{\xi\in U}\widetilde{w}_m(\xi)= \Vert c\Vert,
\]
and consequently $Mw_m(x)\leq 2\Vert c\Vert$.

Second, equicontinuity of $(Mw_m)$ follows directly from the inequality
\begin{equation}\label{eq:equicontinuity}
|M w_{m}(x)-Mw_{m}(y)|\leq \sup_{\xi\in U}|c(x,\xi)-c(y,\xi)|, \quad x,y\in E, m\in \bN
\end{equation}
and continuity of the shift-cost function $c$. Thus, we can use Arzel\`a-Ascoli Theorem. Namely, for any $N\in\bN$ and compact set $B(N):=\{x\in E\colon \| x\|\leq N\}$,
we can find a subsequence of $(Mw_m)$, say $(Mw_{m^N_n})$, and $\phi_N\in C(E)$ such that
\begin{equation}\label{eq:diagonal2}
Mw_{m^N_n}\to \phi_N,\quad n\to\infty,
\end{equation}
uniformly on $B(0,N)$.

Third, using diagonal argument, we show that the limit may be chosen independently of $N$. Indeed, using recursive procedure and taking consecutive subsequences $\{m_n^N\}\subseteq \{m_n^{N-1}\}$, we can find a sequence $(\phi_N)$, such that $\phi_{N}(x)=\phi_{N-1}(x)$ for $x\in B(N-1)$.
Next, for any $x\in E$, we set $N_x=\inf\{ N\in\bN: x\in B(N)\}$ and define $\phi\in C(E)$ by $\phi(x):=\phi_{N_x}(x)$.
Also, for the diagonal sequence $(m_n)$ given by $m_n:=m_n^n$, we get
\begin{equation}
Mw_{m_n}\to \phi,\quad n\to\infty,
\end{equation}
uniformly on $B(0,N)$, for any $N\in\bN$. Using~\eqref{eq:equicontinuity}, we get that $\phi$ satisfies
\begin{equation}\label{eq:lm:limit_function:span}
|\phi(x)-\phi(y)|\leq \sup_{\xi\in U}|c(x,\xi)-c(y,\xi)|, \quad x,y\in E.
\end{equation}

Finally, we show that the convergence $Mw_{m_n}\to \phi$ is (globally) uniform. Let $\vep>0$. From~\eqref{assumpt:cost_function} we know that there exists $N_\vep\in \mathbb{N}$ such that
\begin{equation}\label{eq:bound_sup_c}
\sup_{\xi \in U}|c(x,\xi)-c(y,\xi)|\leq \frac{\vep}{3}, \quad x,y\notin B(N_\vep).
\end{equation}
Since $M w_{m_n}\to \phi$ uniformly on $B(N_\vep+1)$, it is sufficient to show that
\begin{equation}\label{eq:bound_outside}
\sup_{x\notin B(N_\vep)} |M w_{m_n}(x)-\phi(x)|\to 0, \quad n\to\infty.
\end{equation}
Consider $x\notin B(N_\vep)$ and $y\in B(N_\vep+1)\setminus B(N_\vep)$. Recalling ~\eqref{eq:equicontinuity} and~\eqref{eq:lm:limit_function:span}, we get
\begin{align}\label{eq:bound_outside2}
|M w_{m_n}(x)-\phi(x)| & \leq |M w_{m_n}(x)-M w_{m_n}(y)|+|M w_{m_n}(y)-\phi(y)|+|\phi(x)-\phi(y)| \nonumber\\
& \leq |M w_{m_n}(y)-\phi(y)| + 2 \sup_{\xi\in U}|c(x,\xi)-c(y,\xi)|.
\end{align}
Since $y\in B(N_\vep+1),$ starting from some $n_0\in \mathbb{N}$, we get $|M w_{m_n}(y)-\phi(y)|\leq \frac{\vep}{3}$ for $n\geq n_0$. As the choice of $\epsilon$ was arbitrary, this inequality together with~\eqref{eq:bound_sup_c} and~\eqref{eq:bound_outside2} concludes the proof of~\eqref{eq:bound_outside} and consequently Step 1 of the proof.

\medskip
\noindent {\it Step 2.}  For brevity, we drop the subscript $n$ from the diagonal sequence $(m_n)$ given in Step 1, i.e. we assume that $Mw_m\to \phi$ uniformly. We show that
\begin{equation}\label{eq:step2.last}
\phi(x)= Mw(x),\quad x\in E,
\end{equation}
where $w\colon E\to \bR$ is given by
\begin{equation}\label{eq:w_def}
e^{w(x)}:=\inf_{\tau\in \mathcal{T}}\E_x\left[e^{\int_0^{\tau} (f(X_s)-\lambda)ds+\phi(X_{\tau})}\right].
\end{equation}
As in the proof of Proposition~\ref{pr:r(f)_cases}, we use the change of measure technique to transform the problem~\eqref{eq:w_def} into the setting where the assumptions of Theorem~\ref{th:stopping_stability} are satisfied. For any $m\in \mathbb{N}$ and $x\in E$ let us define
\begin{align*}
G_m(x) & :=  Mw_{m}(x)-v(x)+\Vert Mw_{m}\Vert+\Vert v\Vert, \quad & G(x)  & :=  \phi(x)-v(x)+\Vert \phi \Vert+ \Vert v\Vert,\\
d_m & :=  r(f)-\lambda_{m}, \quad & d & :=   r(f)-\lambda,\\
\hat{w}_{m}(x) & :=w_{m}(x)-v(x)+\Vert Mw_{m}\Vert+\Vert v\Vert,\quad & \hat{w}(x)& :=w(x) -v(x)+\Vert \phi\Vert+\Vert v\Vert.
\end{align*}
Observe that $G_m\to G$ uniformly and $d_m\uparrow d$. Moreover, from the assumptions, we get $d>0$. Also, using~\eqref{rm:changed_measure}, Proposition~\ref{pr:r(f)_cases}, and~\eqref{eq:w_def}, we get
\[
e^{\hat{w}_{m}(x)} =\inf_{\tau \in \mathcal{T}_{m}}\E_x^{\mathbb{Q}}e^{\tau d_{m} + G_m(X_{\tau})}\quad \textrm{and}\quad e^{\hat{w}(x)}  =\inf_{\tau \in \mathcal{T}}\E_x^{\mathbb{Q}}e^{\tau d + G(X_{\tau})}.
\]
Hence, using Theorem~\ref{th:stopping_stability}, we get $\hat{w}_{m}\to \hat{w}$ uniformly on compact sets and consequently $w_m\to w$ uniformly on compact sets. Moreover, from Theorem~\ref{th:stopping_continuity}, we get $\hat{w}\in C(E)$ and consequently $w\in C(E)$. Finally, recalling~\eqref{eq:def_M}, we get
\[
\sup_{x\in E}|Mw_{m}(x)-Mw(x)|\leq \sup_{\xi \in U}|w_{m}(\xi)-w(\xi)|\to 0, \quad m\to\infty.
\]
This implies uniform convergence of $Mw_{m}\to Mw$. Also, from Step 1 we know that $Mw_{m}\to \phi$, so $\phi(x)=Mw(x)$, $x\in E$. Recalling~\eqref{eq:w_def}, we conclude the proofs of~\eqref{eq:opt_stop} and Step 2.
\end{proof}
\begin{remark}\label{rm:martingale_impulse_Bellman}
The optimal stopping time for~\eqref{eq:opt_stop} in the case that $\lambda<r(f)$ is given by $\hat\tau=\inf\{t\geq 0: w(X_t)=Mw(X_t)\}$. Moreover, one could easily show that $y(t)=e^{\int_0^{t} (g(X_s)-\lambda) ds +w(X_{ t})}$, $t\geq 0$,
is a submartingale while $(y(t\wedge\hat \tau))$ is a martingale. Those conclusions follow from Theorem~\ref{th:stopping_continuity} by applying reasoning as in the proof of Theorem~\ref{th:existence_Bellman_continuous}.
\end{remark}
Before we present the main contribution of this paper, we need to introduce an auxiliary result for the finite time horizon version of~\eqref{eq:RSC_problem_min}. Let us consider
\begin{equation}\label{eq:finite_horizon}
J_T(x,V):=\ln \E_{(x,V)} \left[\exp\left({\int_0^T f(X_s) ds + \sum_{i=1}^{\infty} 1_{\{\tau_i\leq T\}}  c(X_{\tau_i^-},\xi_i)}\right)\right],
\end{equation}
where $T\geq 0$, $x\in E$, $V\in \bV$, and the remaining notation is aligned with~\eqref{eq:RSC_problem_min}; note that for simplicity we dropped the normalizing constant $1/T$ in \eqref{eq:finite_horizon}. For any $n,m\in\bN$, let $\bV^n\subset \bV$ and $\bV^n_m\subset \bV_m$ denote the families of impulse control strategies and $\delta_m$-dyadic impulse control strategies that have at most $n$ impulses, respectively. In Proposition~\ref{pr:finite_horizon_convergence} we show that the optimal value of \eqref{eq:finite_horizon} could be approximated using strategies from $\bV^n_m$. 

\begin{proposition}\label{pr:finite_horizon_convergence}
For any $T\in \bN$ and $x\in E$ it follows that
\begin{equation}\label{eq:final1}
\lim_{n\to\infty}\lim_{m\to\infty}\inf_{V\in \bV_m^n} J_T(x,V)=\inf_{V\in \bV} J_T(x,V).
\end{equation}
\end{proposition}

\begin{proof}
The proof follows from lemmas presented in Appendix~\ref{S:finite_proof}. Indeed, combining Lemma~\ref{lm:finite_horizon_Bellman} and Lemma~\ref{lm:finite_horizon_conv} we get
\[
\lim_{m\to\infty}\inf_{V\in \bV_m^n} J_T(x,V)=\lim_{m\to\infty}w^n_m(0,x)=w^n(0,x)=\inf_{V\in \bV^n} J_T(x,V).
\]
Next, letting $n\to\infty$ and recalling Lemma~\ref{lm:finite_horizon_conv2}, we conclude the proof.
\end{proof}
Finally, we are ready to link the constant $\lambda$ given by~\eqref{eq:lambda} with the optimal strategy and the optimal value of~\eqref{eq:RSC_problem_min}; see Theorem~\ref{th:opt_strategy_continuous}. In the case $\lambda<r(f)$, Theorem~\ref{th:existence_Bellman_continuous} guarantees the existence of a solution to Bellman equation~\eqref{eq:opt_stop} that might be used to construct the optimal impulse control strategy. In the degenerate case $\lambda=r(f)$, due to monotonicity of $(\lambda_m)$, we know that $\lambda_m=r(f)$, for any $m\in \mathbb{N}$. Thus, from Proposition~\ref{pr:r(f)_cases}, it immediately follows that  the {\it no impulse} strategy is optimal for any dyadic time-grid. As expected, we show that this implies optimality of the  {\it no impulse} strategy in the continuous case. Here, we made additional assumption $E=U$, which allowed us to use finite time horizon results from Proposition~\ref{pr:finite_horizon_convergence}.

Before we present the main result of this paper, we need to introduce supplementary notation. Let us fix $w\in C(E)$ that is a solution to~\eqref{eq:opt_stop} satisfying $w\geq 0$; from Theorem~\ref{th:existence_Bellman_continuous} we know that such $w$ exists if $\lambda<r(f)$. Let $\hat{V}:=(\hat\tau_i,\hat\xi_i)_{i=1}^{\infty}$ be a strategy given recursively by
\begin{equation}\label{eq:last}
\begin{cases}
\displaystyle \hat\tau_{i}  :=\inf\{t\geq \hat\tau_{i-1}: w(X_t)=Mw(X_t)\},\\
\displaystyle \hat\xi_{i} := \argmin_{\xi\in U} \left[c(X_{\hat\tau_{i}^-},\xi)+w(\xi)\right],
\end{cases}
\end{equation}
for $i=1,2,\ldots$, where $\hat\tau_0:=0$. Note that in \eqref{eq:last} a slight abuse of notation was used: for each recursive step, $X_t$ refers to a process for which $i$th impulse is not yet applied. More formally, \eqref{eq:last} could be rewritten as $\hat\tau_i=\hat\sigma \circ \Theta_{\hat\tau_{i-1}}+\hat\tau_{i-1}$, where $\hat\sigma:=\inf\{t\geq 0: w(X_t)=Mw(X_t)\}$ and $\Theta$ is Markov shift operator; see e.g.~\cite[Section 1.4.3]{Shi1978} for details.

\begin{theorem}\label{th:opt_strategy_continuous}
Let $\lambda$ be given by~\eqref{eq:lambda}. Then,
\begin{enumerate}
\item If $\lambda<r(f)$, then
$\lambda=\inf_{V\in \mathbb{V}} J(x,V)$
for any $x\in E$, and the strategy defined in \eqref{eq:last} via Bellman equation~\eqref{eq:opt_stop} is optimal.
\item If $\lambda=r(f)$ and $E=U$, then
$\lambda=\inf_{V\in \mathbb{V}} J(x,V)$
for any $x\in E$, and the {\it no impulse} strategy is optimal.
\end{enumerate}
\end{theorem}

\begin{proof}
For transparency, we split the proof into two parts: (1) when $\lambda<r(f)$; (2) when $\lambda=r(f)$ and $E=U$.

\medskip

\noindent {\it Proof of (1).} We adapt the arguments from \cite[Proposition 2.3]{SadSte2002} to the continuous time case. Given strategy $V\in\bV$, for any $i \in \bN$, we use the notation
\[
X_t^i:=1_{\{t<\tau_{i+1}\}} X_t + 1_{\{t\geq \tau_{i+1}\}}X_{\tau_{i+1}^-},
\]
where $X_{\tau_{i+1}^-}$ denotes the state of the process $(X_t)$ right before the $(i+1)$th impulse.

First, we show that $\lambda=J(x,\hat{V})$ for any $x\in E$. By Remark~\ref{rm:martingale_impulse_Bellman} we know that
\[
e^{\int_0^{\hat\tau_1\wedge T} (f(X_s)-\lambda )ds  +w(X^0_{\hat\tau_1\wedge T})},\quad T\geq 0,
\]
is a martingale under $\bP_{(x,\hat V)}$. As $w(X^0_{\hat\tau_1})=Mw(X^0_{\hat\tau_1})=c(X^0_{\hat\tau_1},\hat\xi_1)+w(\hat\xi_1)$, we get
\begin{align*}
e^{w(x)} & = \E_{(x,\hat{V})} \left[ e^{\int_0^{\hat\tau_1\wedge T} (f(X_s)-\lambda )ds  +w(X^0_{\hat\tau_1\wedge T})}\right]\\
 & =\E_{(x,\hat{V})} \left[ e^{\int_0^{\hat\tau_1\wedge T} (f(X_s)-\lambda )ds  +1_{\{\hat\tau_1\leq T\}} c(X^0_{\hat\tau_1},X^1_{\hat\tau_1})+1_{\{\hat\tau_1\leq T\}}w(X^1_{\hat\tau_1})+1_{\{\hat\tau_1> T\}} w(X_T^0) }\right].
\end{align*}
Acting recursively we get
\begin{equation}\label{eq:last1}
e^{w(x)}  =\E_{(x,\hat{V})} \left[ e^{\int_0^{\hat\tau_n\wedge T} (f(X_s)-\lambda )ds  +\sum_{i=1}^n 1_{\{\hat\tau_i\leq T\}} c(X^{i-1}_{\hat\tau_i},X^i_{\hat\tau_i})+w(X^n_{\hat\tau_n}(T))}\right],
\end{equation}
where
\[
X^n_{\tau_n}(T):=
\begin{cases}
X^n_{\tau_n}, & \tau_n\leq T,\\
X^k_T, & \tau_k\leq T< \tau_{k+1}.
\end{cases}
\]
Using Fatou Lemma and boundedness of $f$ and $w$, we get
\begin{align*}
\infty > e^{w(x)}\geq &  \E_{(x,\hat{V})} \left[\liminf_{n\to\infty} e^{\int_0^{\hat\tau_n\wedge T} (f(X_s)-\lambda )ds  +\sum_{i=1}^n 1_{\{\hat\tau_i\leq T\}} c(X^{i-1}_{\hat\tau_i},X^i_{\hat\tau_i})+w(X^n_{\hat\tau_n}(T))}\right] \\
\geq & \E_{(x,\hat{V})} \left[ e^{T(-\Vert f\Vert-|\lambda|)+\sum_{i=1}^\infty 1_{\{\hat\tau_i\leq T\}} c(X^{i-1}_{\hat\tau_i},X^i_{\hat\tau_i}) -\Vert w \Vert}\right].
\end{align*}
This implies
\[
\E_{(x,\hat{V})} \left[e^{\sum_{i=1}^\infty 1_{\{\hat\tau_i\leq T\}} c(X^{i-1}_{\hat\tau_i},X^i_{\hat\tau_i})}\right]<\infty,
\]
for any $T\geq 0$. Thus, recalling that by~\eqref{A2} the cost function $c$ is bounded away from zero, we conclude that $\hat\tau_n\uparrow \infty$. Consequently, we get $X^n_{\hat\tau_n}(T)\to \widetilde{X}_T$, where $\widetilde{X}_T(\omega):=X^k_T(\omega)$ on $\{\hat\tau_k(\omega)\leq T <\hat\tau_{k+1}(\omega)\}$. Finally, by bounded convergence theorem applied to \eqref{eq:last1}, we get
\begin{equation}\label{eq:th:opt_strategy_continuous:eq1}
e^{w(x)} =\E_{(x,\hat{V})} \left[ e^{\int_0^{T} (f(X_s)-\lambda )ds  +\sum_{i=1}^\infty 1_{\{\hat\tau_i\leq T\}} c(X^{i-1}_{\hat\tau_i},X^i_{\hat\tau_i})+w(\widetilde{X}_T)}\right].
\end{equation}
Taking logarithm of both sides, dividing by $T$, and recalling that $0\leq w(\cdot)\leq \Vert w\Vert$, we get
\[
\lambda = \limsup_{T\to\infty} \frac{1}{T}\ln \E_{(x,\hat{V})} \left[ e^{\int_0^{T} f(X_s)ds  +\sum_{i=1}^\infty 1_{\{\hat\tau_i\leq T\}} c(X^{i-1}_{\hat\tau_i},X^i_{\hat\tau_i})}\right],
\]
which concludes the proof of equality $\lambda=J(x,\hat{V})$, for any $x\in E$.

Second, we show that for any $x\in E$ and strategy $V=(\tau_i,\xi_i)\in \mathbb{V}$, we get $\lambda\leq J(x,V).$ Clearly, we can restrict our attention to the strategies for which
\begin{equation}\label{eq:th:opt_strategy_continuous:eq2}
\E_{(x,V)} \left[ e^{\sum_{i=1}^\infty 1_{\{\tau_i\leq T\}} c(X^{i-1}_{\tau_i},\xi_i)}\right]<\infty,\quad T\geq 0.
\end{equation}
By Remark~\ref{rm:martingale_impulse_Bellman} we know that $e^{\int_0^{T} (f(X_s)-\lambda) ds +w(X_{ T})}$ is a submartingale. Hence,
\[
e^{w(x)}\leq \E_{(x,V)} \left[ e^{\int_0^{\tau_1 \wedge T} (f(X_s)-\lambda) ds +w(X^0_{\tau_1 \wedge T})}\right].
\]
Using the fact that $w(X^{n-1}_{\tau_n})\leq Mw(X^{n-1}_{\tau_n})\leq c(X^{n-1}_{\tau_n},\xi_n)+w(\xi_n)$, we get
\[
e^{w(x)}\leq \E_{(x,V)} \left[ e^{\int_0^{\tau_n \wedge T} (f(X_s)-\lambda) ds +\sum_{i=1}^n 1_{\{\tau_i\leq T\}} c(X^{i-1}_{\tau_i},X^i_{\tau_i})+w(X^n_{\tau_n}(T))}\right].
\]
Recalling~\eqref{eq:th:opt_strategy_continuous:eq2} and letting $n\to\infty$, we get
\[
e^{w(x)}\leq \E_{(x,V)} \left[ e^{\int_0^{T} (f(X_s)-\lambda) ds +\sum_{i=1}^\infty 1_{\{\tau_i\leq T\}} c(X^{i-1}_{\tau_i},X^i_{\tau_i})+\Vert w \Vert}\right].
\]
As before, we get $\lambda\leq J(x,V)$, which concludes this part of the proof.

\medskip

\noindent {\it Proof of (2).} Using Theorem~\ref{th:PS_existence} and Proposition~\ref{pr:r(f)_cases} we get that the cost of the {\it no impulse} strategy equals $r(f)$. Thus, it is sufficient to show that for any $x\in E$ we get $\inf_{V\in \bV} J(x,V)\geq r(f)$. For the contradiction, suppose that $\inf_{V\in \bV} J(x_0,V)< r(f)$ for some $x_0\in E$. Then, for some $\varepsilon>0$, we get
\begin{equation}\label{eq:ccal1}
\limsup_{T\to\infty} \frac{1}{T}\inf_{V\in \bV} J_T(x_0,V) \leq \inf_{V\in \bV} J(x_0,V) <r(f)-\varepsilon,
\end{equation}
where $J_T$ is given by~\eqref{eq:finite_horizon}. Next, we can find $T_0\in \bN$ big enough to get
\begin{equation}\label{eq:ccal2}
\inf_{V\in \bV} \frac{1}{T_0} J_{T_0}(x_0,V)\leq \limsup_{T\to\infty} \frac{1}{T}\inf_{V\in \bV} J_T(x_0,V)+\frac{\varepsilon}{4}\quad \text{and}\quad \frac{\Vert c \Vert}{T_0}\leq \frac{\varepsilon}{4}.
\end{equation}
Using Propostion~\ref{pr:finite_horizon_convergence} we can find $n\in \bN$, $m\in \bN$, and a strategy $\bar{V}\in \bV^n_m$, such that
\begin{equation}\label{eq:ccal3}
\frac{1}{T_0}J_{T_0}(x_0,\bar{V})\leq \inf_{V\in \bV} \frac{1}{T_0}J_{T_0}(x_0,V)+\frac{\varepsilon}{4}.
\end{equation}
Define the strategy $\widetilde{V}$ in the following way: for any period $[kT_0,(k+1)T_0]$, $k\in \bN$, we follow the strategy $\bar{V}$ and at $(k+1)T_0$ we shift the process to $x_0$. It should be noted that $\widetilde{V}\in \bV_m$, as $E=U$ and $T_0\in \bN$. Then, we get
\begin{align}\label{eq:ccal4}
J(x_0,\widetilde{V}) &=\limsup_{k\to\infty} \frac{1}{kT_0} \ln \left(\bE_{(x_0,\widetilde{V})} \left[e^{\int_0^{T_0} f(X_s) ds+\sum_{i=1}^n 1_{\{\tau_i\leq T_0\}}c(X_{\tau_i}^-,\xi_i)+c(X_{T_0},x_0)}\right]\right)^k\nonumber\\
& \leq \frac{1}{T_0}J_{T_0}(x_0,\bar{V}) + \frac{\Vert c \Vert}{T_0}.
\end{align}
Combining~\eqref{eq:ccal1}--\eqref{eq:ccal3} with \eqref{eq:ccal4} we get $J(x_0,\widetilde{V})<r(f)-\frac{\varepsilon}{4}$. This leads to contradiction since by Proposition~\ref{pr:r(f)_cases} we have $\inf_{V\in \bV_m}J(x_0,V)=r(f)$.
\end{proof}

Theorem \ref{th:opt_strategy_continuous} has an important meaning from applications' point of view. It says that if $\lambda<r(f)$, then $\lambda$ is the optimal value of the cost functional  $J(x,V)$. On the other hand, optimal values $(\lambda_m)$ of the discretized dyadic impulse control problems approximate $\lambda$ from above, i.e. $\lambda_m$ converges decreasingly to $\lambda$. Thus, an $\varepsilon$-optimal strategy for $\lambda_m$ is also an $(\varepsilon + \lambda_m-\lambda)$-optimal strategy for the continuous time impulse control. This shows that the dyadic impulse control strategies could be used to obtain nearly optimal strategies for the original continuous time impulse control. Moreover, dyadic impulse control can be subject to further state discretization. To sum up, this shows how one can construct feasible nearly optimal controls for the original continuous time impulse control problem.

Alternatively, if $\lambda<r(f)$, one could try to directly approximate Bellman equation \eqref{eq:opt_stop} to construct nearly optimal strategies. However, we do not have uniqueness (up to an additive constant) of $w$, i.e. we were only able to show that $w_m$ converges to $w$ uniformly along a subsequence.\footnote{Still, recall that optimal impulse control for dyadic Bellman equation \eqref{eq:opt_stop_discrete} is nearly optimal for Bellman equation \eqref{eq:opt_stop} as $\lambda$ is uniquely defined due to monotonicity of the sequence $(\lambda_m)$.} If we assume that $w$ is unique, we get that $w_m$ converges to $w$ uniformly and consequently we can control construction of nearly optimal strategies using a direct approximation of Bellman equation \eqref{eq:opt_stop}. This is the case e.g. when $c(x,\xi)=c(x)+d(\xi)$; see~\cite[Section 3]{SadSte2002} for details.

\section{Reference examples}\label{S:examples}

In this section we show three examples of processes satisfying assumptions~\eqref{A3}--\eqref{A4}. Example~\ref{ex:1} is taken from \cite{PitSte2019} and considers a piecewise deterministic process from \cite{Dav1984} that was later embedded into the control theory framework in~\cite{BauRie2011}. Example~\ref{ex:2} considers the compact domain reflected diffusion process that was studied in~\cite{GarMen2002} and~\cite{MenRob1997}. Example~\ref{ex:3} could be see as an extension of Example~\ref{ex:2} allowing jumps; see~\cite[Remark 2.1b]{MenRob2018} and references therein. For brevity, we decided to only outline the frameworks without providing formal process construction details; we refer to \cite{PitSte2019,BauRie2011,GarMen2002,MenRob2018} for further information.

\begin{example}[Piecewise deterministic process]\label{ex:1}
Assume that $(X_t)$ is a piecewise deterministic process. The deterministic part is a solution to a stable differential equation
\begin{equation}\label{eqq1}
dX_t=F(X_t)dt,
\end{equation}
with initial state $X_0=x$. The process follows this dynamics till (random) jump moment, and then is subject to immediate shift, after which its evolution follows the same deterministic logic till next jump occurs, and so on. We assume that the sequence of jumps, say $(\tau_n)$, is such that $(\tau_{n+1}-\tau_n)$ is i.i.d. and exponentially distributed with fixed intensity $r>0$. The shifts are made according to transition measure such that
\[
X_{\tau_n}=A(X_{\tau_n^-})+w_n,
\]
where $w_n$ is a sequence of i.i.d. standard normal random variables and  function $A\colon \bR\to \bR$ is bounded. Assuming suitable regularity of $F$, for any $t<\tau_1$ and initial state $x$, we get $X_t=\phi(x,t)$, where $\phi$ is a continuous function. Moreover, we assume that $\phi$ is such that for any $x\in E$ we have $|\phi(x,t)|\leq e^{-\alpha t}|x|+M$, where $\alpha, M>0$ are some predefined constants that are independent of $x$. In particular, this implies \eqref{A3}. Also, following \cite[Example 5.2]{PitSte2019}, we can show that \eqref{A4a} is satisfied. Indeed, noting that after the $1$st jump the process starts from $A(X_{\tau_1^-})+w_1$, where $A$ is bounded  and $w_1$ is standard normal, we get
\begin{equation}
\kappa_t:=\sup_{x,y\in E}\sup_{A\in \cal{E}} \mathbb{P}_x(X_t\in A)-\mathbb{P}_y(X_t\in A)<1,\quad \textrm{for } t>0,
\end{equation}
which implies \eqref{A4a}. Next, we get that $T_tv(x):=\ln \mathbb{E}_x \left[\exp \left(\int_0^t f(X_s)ds + v(X_t)\right)\right]$ is a local contraction operator in the span norm $\|\cdot\|_{sp}$; see  \cite{PitSte2019} for details. Thus, for any function $f\in C(E)$ satisfying $\kappa_t<e^{-t \|f\|_{sp}}$ for some $t>0$, we get \eqref{A4b}.
\end{example}

\begin{example}[Reflected diffusion]\label{ex:2} Let $\alpha\in (0,1)$ and $E:=\overline{\mathcal{O}}$, where $\mathcal{O}$ is a bounded non-empty domain in $\mathbb{R}^d$ with $\mathcal{C}^{2+\alpha}$-class boundary. Also, we assume that $A=(a_{ij})_{i,j=1}^d$ is a uniformly elliptic and symmetric matrix-valued mapping, with bounded and ($\alpha$-exponent) H\"{o}lder continuous margins $a_{ij}:\overline{\mathcal{O}}\mapsto \mathbb{R}$. Finally, let $b:=(b_i)_{i=1}^{d}$ be an  $\mathbb{R}^d$-valued mapping with $\mathcal{C}^{1+\alpha}(\partial \mathcal{O})$-class margins that satisfy non-tangentiality condition, i.e. for some $c_0>0$ we get $b(x)\cdot n(x)\geq c_0$, $ x\in \partial \mathcal{O}$, where $n(x)=(n_1(x),\ldots,n_d(x))$ denotes the unit outward normal to $\mathcal{O}$ at $x$.

Following argument leading to~\cite[Equation (7.1.18)]{GarMen2002}, we get that there exists a weak solution to
\begin{equation}\label{eq:reflected_diff}
dX_t=A^{1/2}(X_t)dW_t-b(X_t)d\xi_t,
\end{equation}
where $A^{1/2}$ denotes the positive square root of $A$. Namely, there exists a pair of processes $(X_t,\xi_t)$ for some $d$-dimensional Brownian motion $(W_t)$, where the process $(X_t)$ is understood as the reflected diffusion and $(\xi_t)$ describes the reflection; see~\cite[Section 2.1]{BenLio1984} for details. From~\cite[Section 7.1]{GarMen2002}, we know that $(X_t)$ is Feller--Markov with values in $E$ and the transition probability satisfying
\begin{equation}\label{eq:reflected_diff_density1}
\mathbb{P}_x\left[X_t\in A\right]=\int_A p_t(x,y) dy, \quad t>0, x\in E, A\in \mathcal{B}(E),
\end{equation}
where $(x,y)\mapsto p_t(x,y)$ is continuous for any $t>0$. Also, using \cite[Theorem 4.3.7]{GarMen2002} and compactness of $E$, for any $t>0$ we can find constants $0<a_t\leq b_t<\infty$ such that $a_t\leq p_t(x,y)\leq b_t$, for $x,y\in E$, hence \eqref{A4} is satisfied. Recalling that $X$ is $C$-Feller and $E$ is compact, we also get \eqref{A3}; see~\cite[Lemma 5.1 and Propositon 6.4]{BasSte2018}. 

\end{example}

\begin{example}[Reflected diffusion with jumps]\label{ex:3}
Using the results from the theory of integro-differential equations one could expand Example~\ref{ex:2} by allowing jumps. In~\cite{GarMen2002}, it is shown that the solution to
\begin{equation}\label{eq:ex:2}
dX_t=A^{1/2}(X_t)dW_t+\int_{\mathbb{R}^d\setminus\{0\}}z\mu_X(dt,dz)-b(X_t)d\xi_t,
\end{equation}
exists. Here, $\mu_X$ denotes the measure associated with the Doob-Meyer decomposition of the suitable L\'{e}vy measure; we refer to \cite[Section 7.1]{GarMen2002} for details. Due to the second term in \eqref{eq:ex:2}, the process $(X_t)$ might be seen as a reflected diffusion with jumps. Using similar logic as in Example~\ref{ex:2}, one can show that~\eqref{A3} and~\eqref{A4} are satisfied.
\end{example}

\appendix

\section{Properties of the optimal stopping problems}\label{S:stopping} For completeness, we present selected auxiliary properties of the risk sensitive optimal stopping problems that are used in this paper. This section could be seen as an extract from~\cite{JelPitSte2019a}. For brevity we omit the proofs; see \cite{JelPitSte2019a} for details and further discussions.

Throughout this section we use $X=(X_t)$ to denote an uncontrolled standard Feller-Markov process that satisfy assumption \eqref{A3}; \eqref{A4} is not required. It should be noted that while in~\cite{JelPitSte2019a} it is assumed that the process is $C_0$-Feller, all proofs therein are based on the process properties given in \cite[Proposition 2.1]{JelPitSte2019a}. It is easy to check that those properties are satisfied in our setting due to $C$-Feller property and \eqref{A3}.

Let $G$ and $g$ be real-valued continuous and bounded functions on $E$, such that $G(\cdot)\geq 0$ and $g(\cdot)\geq c$ for some $c>0$. Recall that by $\mathcal{T}$ we denote the family of stopping times with values in $[0,\infty)$, while $\mathcal{T}_{m}$ denotes the subfamily of dyadic stopping times defined on time-grid $\{0,\delta_m, 2\delta_m, \ldots\}$, where $\delta_m:=(1/2)^m$ and $m\in\bN$.

We start with the verification theorem for the optimal stopping problem on a discrete time-grid. For the proof we refer to \cite[Proposition 3.3]{JelPitSte2019a}; see also \cite[Proposition 4.3]{JelPitSte2019a}. This result is helpful in the proof of Proposition~\ref{pr:r(f)_cases}.

\begin{proposition}[{\cite[Proposition 3.3]{JelPitSte2019a}}]\label{pr:stopping_Bellman} Let $m\in \mathbb{N}$ and let $u_m(\cdot)\geq 1$ be a measurable function satisfying
$u_m(x)=\mathbb{E}_x\left[\exp\left(\int_0^{\delta_m}g(X_s)ds\right) u_m(X_{\delta_m})\right]\wedge e^{G(x)}$ for any $x\in E$.
Then, for any $x\in E$, we have
\[
u_m(x)=\inf_{\tau \in \mathcal{T}_m} \mathbb{E}_x\left[\exp\left(\int_0^{\tau}g(X_s)ds+G(X_{\tau})\right)\right].
\]
\end{proposition}
Next, we list the properties of the infinite horizon optimal stopping problem. In particular, we have the continuity of the optimal value function; this is helpful e.g. when proving Theorem~\ref{th:existence_Bellman_continuous}.
\begin{theorem}[{\cite[Theorem 4.7 and Remark 4.9]{JelPitSte2019a}}]\label{th:stopping_continuity}
The function $u$ given by
\[
u(x):=\inf_{\tau \in \mathcal{T}}\mathbb{E}_x\left[\exp\left(\int_0^{\tau}g(X_s)ds+G(X_{\tau})\right)\right], \quad x\in E,
\]
is continuous and bounded. Moreover, the optimal stopping time for $u$ is given by $\hat{\tau}:=\inf\left\{t\geq 0: u(X_t)=e^{G(X_t)}\right\}$, the process $y(t):=e^{\int_0^t g(X_s)ds}u(X_t)$, $t\geq 0$, is a submartingale and $(y(t\wedge\hat{\tau}))$, $t\geq 0$, is a martingale.
\end{theorem}
Now, let $(g_{m})$ and $(G_m)$ be sequences of functions from $C(E)$. Assume that $g_m\uparrow g$, as $m\to \infty$, with $g_0(\cdot)\geq c_0> 0$. Furthermore, let $G_m\to G$ uniformly, as $m\to \infty$, and $G_m(\cdot)\geq 0$, for any $m\geq 0$. Let us define the family of functions
\begin{align}
u_m(x)&:=\inf_{\tau \in \mathcal{T}_m} \mathbb{E}_x\left[\exp\left(\int_0^{\tau}g_m(X_s)ds+G_m(X_{\tau})\right)\right], \quad m\in \bN,\, x\in E,\label{eq:stopping_stability_1}\\
u(x)&:=\inf_{\tau \in \mathcal{T}} \mathbb{E}_x\left[\exp\left(\int_0^{\tau}g(X_s)ds+G(X_{\tau})\right)\right], \quad x\in E.\label{eq:stopping_stability_2}
\end{align}
Next theorem provides a convergence result for the optimal stopping problem. This is used in the proof of Theorem~\ref{th:existence_Bellman_continuous}.
\begin{theorem}[{\cite[Theorem 5.1]{JelPitSte2019a}}]\label{th:stopping_stability}
Let $u_m$ and $u$ be given by~\eqref{eq:stopping_stability_1} and~\eqref{eq:stopping_stability_2}, respectively. Then, $u_m \to u$ as $m\to\infty$ uniformly on compact sets.
\end{theorem}
Now, fix $T\geq 0$ and let the functions $g:[0,T]\times E\mapsto \bR$ and $G:[0,T]\times E\mapsto \bR$ be continuous and bounded.\footnote{Note that here we do not assume nonnegativity of $g$ and $G$. This is because now we act in the finite time horizon setting; see~\cite[Remark 4.5]{JelPitSte2019a}.} For any $t\in [0,T]$, $h\in [0,t]$, and $x\in E$, let
\begin{align}
u(t,x) & :=\inf_{\tau\leq T-t}\bE_x\left[\exp\left(\int_0^\tau g(s+t,X_s)ds+G(\tau+t,X_\tau)\right)\right].\label{eq:time_inh_stop_2}
\end{align}
Next proposition could be seen as an extension of Proposition 4.2 and Proposition 4.3 from~\cite{JelPitSte2019a} to the case when the cost functions depend on the time parameter. This proposition is used in Appendix~\ref{S:finite_proof}.
\begin{proposition}\label{pr:time_dep_stopping}
The function $u$ given by~\eqref{eq:time_inh_stop_2} is jointly continuous and bounded. Moreover, the optimal stopping time for $u(t,x)$ is given by
\begin{equation}\label{optstt}
\tau_{t}:=\inf\left\{s\geq 0:u(t+s,X_s)=e^{G(t+s,X_s)}\right\}.
\end{equation}
\end{proposition}

\begin{proof}
For completeness, since Proposition~\ref{pr:time_dep_stopping} is not proved directly in~\cite{JelPitSte2019a}, we decided to outline the proof. We apply a space enlargement technique to use the results from~\cite{JelPitSte2019a}. Let $(\widetilde{\Omega},\widetilde{\cF},\widetilde{\mathbb{F}})$ be a filtered product space, where $\widetilde{\Omega}:=[0,\infty)\times \Omega$, the $\sigma$-field is given by $\widetilde{\cF}:=\mathcal{B}[0,\infty)\otimes \cF$,  and the filtration $\widetilde{\mathbb{F}}:=(\widetilde{\cF}_t)_{t\in [0,\infty)}$, is defined as $\widetilde{\cF}_t:=\mathcal{B}[0,\infty)\otimes \cF_t$.
Also, let $\widetilde{E}:=[0,\infty)\times E$ and $\widetilde{\mathcal{E}}:=\mathcal{B}[0,\infty)\otimes \mathcal{E}$. Let us define the space-time process $\widetilde{X}=(\widetilde{X}_t)_{t\in [0,\infty)}$ by
\begin{equation}\label{eq:def_space_time}
\widetilde{X}_t(s,\omega):=(t+s,X_t(\omega)), \quad t\geq 0,(s,\omega)\in [0,\infty)\times \Omega.
\end{equation}
Using classic argument one can show that $\widetilde{X}$ is a standard Markov process on the product filtered probability space $(\widetilde{\Omega},\widetilde{\cF},\widetilde{\mathbb{F}},\widetilde{\mathbb{P}})$ that takes values in $(\widetilde{E},\widetilde{\mathcal{E}})$ and satisfies $C_0(\widetilde{E})$-Feller property; the measure $\widetilde{\mathbb{P}}$ is given via the kernel transition functions corresponding to the family $(\widetilde{\mathbb{P}}_{(t,x)})_{(t,x)}$, where $\widetilde{\mathbb{P}}_{(t,x)}:=\delta_t \otimes \mathbb{P}_x$, $\delta_t$ denotes the Dirac measure, $t\in [0,\infty)$, and $x\in E$. Note that $\widetilde{\mathbb{P}}_{(t,x)}$ might be linked to the distribution of $\widetilde{X}$ starting from $(t,x)$. We refer to~\cite[Section 1.4.6]{Shi1978} for a discussion of a similar homogenization procedure.

Note that in this setting, for any $t\in [0,T]$ and $x\in E$, we get
\[
u(t,x):=\inf_{\tau \leq T-t} \widetilde{\bE}_{(t,x)} \left[ \exp\left(\int_0^\tau g(\widetilde{X}_s)ds+G(\widetilde{X}_\tau)\right)\right].
\]
Thus, using~Proposition 4.3, Remark 4.4, and Remark 4.5 from \cite{JelPitSte2019a}, we get that $(t,x)\mapsto u(t,x)$ is jointly continuous and the optimal stopping time for $u(t,x)$ is given by \eqref{optstt}, which concludes the proof.
\end{proof}

\section{Impulse control on finite horizon}\label{S:finite_proof} In this section we study the finite time horizon impulse control problem given in \eqref{eq:finite_horizon}, i.e.
\begin{equation}\label{eq:finite_horizon2}
J_T(x,V):=\ln \E_{(x,V)} \left[\exp\left({\int_0^T f(X_s) ds + \sum_{i=1}^{\infty} 1_{\{\tau_i\leq T\}}  c(X_{\tau_i^-},\xi_i)}\right)\right],
\end{equation}
where $T\geq 0$, $x\in E$, $V\in \bV$, and with the remaining notation aligned with~\eqref{eq:RSC_problem_min}.
The main goal of this section is to show properties that are useful when proving Proposition~\ref{pr:finite_horizon_convergence}. Before we present the results, let us introduce some auxiliary notation. Let us fix $T\in \mathbb{N}$ and define the family $(w^n)_{n\in\bN}$ of functions $w^n\colon [0,T]\times E \to \bR$ given recursively, for $n\in \bN$, by
\begin{align}\label{eq:Bellman_finite}
w^0(t,x)& :=\ln \bE_x\left[\exp\left(\int_0^{T-t} f(X_s)ds\right)\right],\nonumber\\
w^{n}(t,x)& :=\inf_{\tau\leq T-t}\ln \bE_x\left[\exp\left(\int_0^\tau f(X_s)ds+\widetilde{M}w^{n-1}(t+\tau,X_\tau)\right)\right],
\end{align}
where $\widetilde{M}\colon C([0,T]\times E )\to C([0,T]\times E )$ is given by
\[
\widetilde{M} h(t,x):=\min\left(\inf_{\xi\in U}\left(c(x,\xi)+h(t,\xi)\right),h(t,x)\right).
\]
Similarly, for any $m\in \bN$, we define the family $(w_m^n)_{n\in\bN}$ recursively, for $n\in\bN$, by
\begin{align}\label{eq:Bellman_finite_discrete}
w^0_m(t,x)& :=\ln \bE_x\left[\exp\left(\int_0^{T-t} f(X_s)ds\right)\right],\nonumber\\
w^{n}_m(t,x)& :=\inf_{\substack{\tau\leq T-t\\ \tau \in \mathcal{T}^m}}\ln \bE_x\left[\exp\left(\int_0^\tau f(X_s)ds+\widetilde{M}w^{n-1}_m(t+\tau,X_\tau)\right)\right].
\end{align}
In the following, we link $w^n$ and $w^n_m$ to~\eqref{eq:finite_horizon}. Before we do that, we show continuity and the optimal stopping rule for $w^n$.

\begin{lemma}\label{lm:cont_w_hetero}
For any $n\in \bN$, the map $(t,x)\mapsto w^n(t,x)$ defined by~\eqref{eq:Bellman_finite} is jointly continuous. Furthermore, if $n\geq 1$, then the optimal stopping time for $w^n(t,x)$ is given by
\begin{equation}\label{eq:opt_stop_hetero}
\sigma^n(t):=\inf\left\{s\geq 0:w^{n}(t+s,X_s)=\widetilde{M}w^{n-1}(t+s,X_s)\right\}.
\end{equation}
\end{lemma}

\begin{proof} We proceed by induction. First, we prove that the map $(t,x)\mapsto w^0(t,x)$ is jointly continuous. Noting that
\begin{align*}
|\ln z-\ln y| \leq \frac{1}{\min(z,y)}|z-y|, \, z,y>0, \quad \text{and} \quad
|e^z-e^y| \leq e^{\max(z,y)} |z-y|, \, z,y\in\bR,
\end{align*}
and using boundedness of $f$, for $t,s\in [0,T]$, $y\in E$, and $L:=e^{2 T\Vert f\Vert} \Vert f\Vert$, we get
\[
|w^0(t,y)-w^0(s,y)| \leq e^{T\Vert f\Vert}  \bE_y\left| e^{\int_0^{T-t} f(X_u)du}- e^{ \int_0^{T-s} f(X_u)du}\right| \leq  L \vert t-s\vert.
\]
By Feller property (see also Part 4 of~\cite[Proposition 2.1]{JelPitSte2019a}) we get that the map $x\mapsto w^0(t,x)$ is continuous for any fixed $t\in [0,T]$. Thus, for any sequence $((t_k,x_k))_{k\in\bN}$ converging to $(t,x)$, we get
\begin{align*}\label{eq:w^0_m_cont}
|w^0(t_k,x_k)-w^0(t,x)|& \leq |w^0(t_k,x_k)-w^0(t,x_k)|+|w^0(t,x_k)-w^0(t,x)|\to 0,
\end{align*}
as $k\to\infty$.

Second, we show continuity and the optimal stopping time for $w^n$, when $n\geq 1$. Combining induction assumption with compactness of $U$ we get joint continuity of the map $(t,x)\mapsto\widetilde{M}w^{n-1}(t,x)$. Thus, using Proposition~\ref{pr:time_dep_stopping}, we get joint continuity of $(t,x)\mapsto w^n(t,x)$ and optimality of $\sigma^n(t)$ for any $t\in [0,T]$, which concludes the proof.
\end{proof}
Now, we link $w^n$ and $w^n_m$ to the value function $J_T$.
\begin{lemma}\label{lm:finite_horizon_Bellman}
For any $n,m\in \bN$, and $x\in E$, we get $w^n(0,x)=\inf_{V\in \bV^n}J_T(x,V)$ and $w^n_m(0,x)=\inf_{V\in \bV^n_m}J_T(x,V)$.
\end{lemma}

\begin{proof}
For brevity, we present the proof only for $w^n$; the proof for $w^n_m$ is analogous. Clearly, for $n=0$, the claim is straightforward so we can assume that $n\geq 1$.

First, we show that $w^n(0,x)=J_T(x,\hat{V})$, where $\hat{V}\in \bV^n$ is a specific strategy given below. To ease the notation, we set $\bar{c}(x,\xi):=1_{\{x\neq \xi\}} c(x,\xi)$, $x\in E$, $\xi\in U$. Then, recalling that $c(x,\xi)\geq c_0 >0$, we get $\widetilde{M}h(t,x)=\inf_{\xi\in U\cup\{x\}}\left( \bar{c}(x,\xi)+h(t,\xi)\right).$ Recalling~\eqref{eq:opt_stop_hetero}, for any $i=1,2,\ldots, n$, let us define recursively\footnote{For simplicity, we assume that there are unique minimizers in $\widetilde \xi_i$ and later in $\hat \xi_i$.}
\begin{align*}
\widetilde\tau_i & := \sigma^{n-i+1}(\widetilde{\tau}_{i-1})\circ\Theta_{\widetilde{\tau}_{i-1}}+\widetilde{\tau}_{i-1},\\
\widetilde \xi_i & := \argmin_{\xi\in U\cup\{ X_{\widetilde\tau_i^-} \}}\left( \bar{c}(X_{\widetilde\tau_i^-},\xi)+w^{n-i}(\widetilde\tau_i,\xi)\right),
\end{align*}
where $\widetilde\tau_0\equiv 0$ and $\Theta$ denotes Markov shift operator. Next, we define the strategy $\hat{V}:=(\hat\tau_i,\hat\xi_i)_{i=1}^n\in \bV^n$, where
\begin{align*}
\hat\tau_i &:= \inf \left\{ \widetilde\tau_j\geq \hat\tau_{i-1}: \inf_{\xi\in U}\left( c(X_{\widetilde\tau_j^-},\xi)+w^{n-j}(\widetilde\tau_j,\xi)\right)<w^{n-j}(\widetilde\tau_j,X_{\widetilde\tau_j^-})\right\},\\
\hat\xi_i &:= \argmin_{\xi\in U}\left( c(X_{\hat\tau_i^-\wedge T},\xi)+w^{n-i}(\hat\tau_i\wedge T,\xi)\right),
\end{align*}
is defined recursively with $\hat\tau_0\equiv 0$; we adapt the standard convention $\inf\{\emptyset\}=+\infty$. Note that $(\hat \tau_i)$ is a modification of $(\widetilde \tau_i)$ which takes into account only the situations where the process is shifted.
Recalling Lemma~\ref{lm:cont_w_hetero} and using recursive arguments combined with strong Markov property, we get
\begin{align*}
w^n(0,x)&=\ln \E_{(x,\hat{V})}\left[\exp\left(\int_0^{\widetilde{\tau}_1} f(X_s)ds+\bar{c}(X_{\widetilde\tau_1^-},\widetilde{\xi}_1)+w^{n-1}(\widetilde\tau_1,\widetilde\xi_1)\right)\right]\\
& =\ln \E_{(x,\hat{V})}\left[\exp\left(\int_0^{T} f(X_s)ds+\sum_{i=1}^n\bar{c}(X_{\widetilde\tau_i^-},\widetilde{\xi}_i)\right)\right]\\
& =\ln \E_{(x,\hat{V})}\left[\exp\left(\int_0^{T} f(X_s)ds+\sum_{i=1}^n 1_{\{\hat\tau_i\leq T\}}c(X_{\hat\tau_i^-},\hat{\xi}_i)\right)\right],
\end{align*}
which concludes the proof of $w^n(0,x)=J_T(x,\hat{V})$.

Second, using similar argument one can show that $w^n(0,x)\leq \inf_{V\in \bV^n}J_T(x,V)$; see~\cite[Proposition 2.3]{SadSte2002} for details. Thus, $w^n(0,x)= \inf_{V\in \bV^n}J_T(x,V)$, which concludes the proof.
\end{proof}
Next, we show that the impulse control problem with finitely many impulses may be approximated by the optimization problems on the discrete time-grid.
\begin{lemma}\label{lm:finite_horizon_conv}
Let $n\in \bN$ and $\Gamma\subset E$ be a compact set. Then $w^n_m(t,x)\to w^n(t,x)$ as $m\to \infty$ uniformly in $(t,x)\in[0,T]\times \Gamma$.
\end{lemma}

\begin{proof}
Let us fix a compact set $\Gamma\subset E$ and proceed by induction. The claim for $n=0$ is straightforward as $w^0_m \equiv w^0$ for any $m\in \bN$. Let $n\in \bN$ and assume that the assertion holds for $w^n$; we show it for $w^{n+1}$. The proof is based on the argument used in~\cite[Lemma 4.1]{JelPitSte2019a}.

For any $m\in \bN$, $t\in [0,T]$ and $x\in E$, we set
\[
h^{n+1}(t,x):=\exp \left(w^{n+1}(t,x)\right),\quad h^{n+1}_m(t,x):=\exp\left( w_{m}^{n+1}(t,x)\right).
\]
Noting that $\min\left( h^{n+1}(t,x),h^{n+1}_m(t,x)\right)\geq e^{-T(n+2)\Vert f\Vert}$ for any $m\in \bN$, $t\in [0,T]$, and $x\in E$, and using inequality $|\ln y-\ln z|\leq \frac{1}{\min(y,z)}|y-z|$, for $y,z> 0$, we get
\begin{equation}\label{eq:acalc0}
0  \leq  w^{n+1}_m(t,x)-w^{n+1}(t,x)\leq e^{T(n+2)\Vert f\Vert} (h^{n+1}_m(t,x)-h^{n+1}(t,x)).
\end{equation}
Thus, it is enough to show that $h^{n+1}_m(t,x)\to h^{n+1}(t,x)$ uniformly on $[0,T]\times\Gamma$. Before we do that, we need to introduce some auxiliary notation and results.

Consider any $t\in [0,T]$ and $x\in \Gamma$. Let $\varepsilon>0$ and $\tau^{(\varepsilon,t,x)}\leq t$ be an $\varepsilon$-optimal stopping time for $h^{n+1}(t,x)$. For any $m\in \bN$, we set
\[
\tau_m^{(\varepsilon,t,x)} :=\textstyle \sum_{j=1}^{\lfloor(T-t)2^m\rfloor}1_{\left\{\frac{j-1}{2^m}< \tau^{(\varepsilon,t,x)}\leq \frac{j}{2^m}\right\}}\frac{j}{2^m}.
\]
In the following, for brevity, we write $\tau$ and $\tau_m$ instead of $\tau^{(\varepsilon,t,x)}$ and $\tau_m^{(\varepsilon,t,x)}$, and for $m\in \bN$, $s,u\in [0,T]$, and $y,z\in E$, we set
\begin{align*}
Z^n(t,\tau) &:=\exp\left(\textstyle \int_0^\tau f(X_s)ds+\widetilde{M} w^n(t+\tau,X_\tau)\right),\\
A^n_m(s,y)& : =|w^n_m(s,y)-w^n(s,y)|,\\
B^n_m(s,u,y,z)&:=\left|\widetilde{M} w^n_m(s,y)-\widetilde{M} w^n(u,z)\right|,\\
C^n(s,u,y,z)& :=|w^n(s,y)-w^n(u,z)|+\sup_{\xi\in U} |c(y,\xi)-c(z,\xi)|+\sup_{\xi\in U} |w^n(s,\xi)-w^n(u,\xi)|.
\end{align*}
Note that, by Lemma~\ref{lm:cont_w_hetero}, the function $C^n$ is jointly continuous and bounded. Moreover, by induction assumption, $A^n_m(s,y)\to 0$ as $m\to \infty$ uniformly in $[0,T]\times \hat\Gamma$ where $\hat\Gamma\subset E$ is compact. Also, for any $m\in \bN$, $s,u\in [0,T]$, and $y,z\in E$, we get
\begin{align}\label{eq:acalc2}
B^n_m(s,u,y,z) & \leq |w^n_m(s,y)-w^n(u,z)|+ \sup_{\xi\in U} |c(y,\xi)-c(z,\xi)|+\sup_{\xi\in U}|w^n_m(s,\xi)-w^n(u,\xi)| \nonumber\\
& \leq  |w^n_m(s,y)-w^n(s,y)|+|w^n(s,y)-w^n(u,z)|+\sup_{\xi\in U} |c(y,\xi)-c(z,\xi)| \nonumber \\
&\phantom{=} +\sup_{\xi\in U}|w^n_m(s,\xi)-w^n(s,\xi)|+\sup_{\xi\in U}|w^n(s,\xi)-w^n(u,\xi)|\nonumber \\
& \leq  A^n_m(s,y) + C^n(s,u,y,z)+\sup_{v\in [0,T]}\sup_{\xi\in U}A^n_m(v,\xi).
\end{align}
Next, recalling \eqref{A3} we can find $R>0$ such that
\begin{equation}\label{eq:acalc3}
\sup_{x\in \Gamma} \bP_x \left[\sup_{s\in [0,T]}\rho(X_s,x)> R\right]\leq \varepsilon.
\end{equation}
Let $B:=\{x\in E\colon\rho(x,\Gamma)\leq R+1\}$. Using induction assumption and compactness of $U$, we may find $m_0\in \bN$ such that for any $m\geq m_0$ we get
\begin{equation}\label{eq:acalc4}
\sup_{s\in [0,T]}\sup_{y\in B\cup U}A^n_m(s,y)\leq \varepsilon.
\end{equation}
Also, recalling $\delta_m=2^{-m}$ and noting that $C^n$ is uniformly continuous on $[0,T]^2\times B \times B$ and $C^n(s,s,y,y)=0$ for any $s\in [0,T]$ and $y\in E$, we may find $r_1>0$ and $m_1\in \bN$ such that for any $m\geq m_1$ we get
\begin{equation}\label{eq:acalc5}
\sup_{\substack{s,u\in [0,T]\\|s-u|\leq \delta_m}} \sup_{\substack{y,z\in B\\\rho(y,z)\leq r_1}} C^n(s,u,y,z) \leq \varepsilon.
\end{equation}
Let $r:=\min(r_1,\frac{1}{2})$. By \eqref{A3} we may find $m_2\in \bN$ such that for any $m\geq m_2$ we get
\begin{equation}\label{eq:acalc6}
\sup_{x\in B} \bP_x\left[ \sup_{s\in [0,\delta_m]}\rho(X_s,x)\geq r\right]\leq \varepsilon \quad \text{and} \quad \Vert f\Vert \delta_m\leq \varepsilon.
\end{equation}
Finally, it is useful to note that $\Vert w^k\Vert\leq T\Vert f\Vert$ and $\Vert w^k_m\Vert\leq T\Vert f\Vert$, for $k, m\in \bN$; these follow from Lemma~\ref{lm:finite_horizon_Bellman} and the fact that the cost of {\it no impulse} strategy is bounded from above by $T\Vert f\Vert$.

Let us now go back to the proof of uniform convergence of $h^{n+1}_m(t,x)\to h^{n+1}(t,x)$. Recalling boundedness of $f$ and using inequality $\vert \tau_m-\tau\vert\leq \delta_m$, for any $t\in [0,T]$ and $x\in \Gamma$, we get
\begin{align}\label{eq:acalc1}
0 &\leq h^{n+1}_m(t,x)-h^{n+1}(t,x)\nonumber \\
 &\leq  \bE_x\left[\exp\left(\int_0^{\tau_m} f(X_s)ds+\widetilde{M} w^n_m(t+\tau_m,X_{\tau_m})\right)\right] -\bE_x\left[Z^n\right] +\varepsilon \nonumber\\
&\leq e^{\Vert f\Vert \delta_m}\bE_x\left[Z^n(t,\tau)e^{B^n_m(t+\tau_m,t+\tau,X_{\tau_m},X_{\tau})}\right] -\bE_x\left[Z^n(t,\tau)\right] +\varepsilon.
\end{align}
Noting that $Z^n(t,\tau)\leq e^{\Vert c\Vert+2T\Vert f\Vert}$ and $\Vert B^n_m\Vert\leq  2\Vert c\Vert+4T\Vert f\Vert$ for any $m\in \bN$, and using~\eqref{eq:acalc2}--\eqref{eq:acalc4}, for any $t\in [0,T]$, $x\in \Gamma$ and $m\geq \max(m_0,m_1,m_2)$, we get
\begin{multline}\label{eq:acalc7}
\bE_x\left[Z^n(t,\tau) e^{B^n_m(t+\tau_m,t+\tau,X_{\tau_m},X_{\tau}))}\right] \\
\leq e^\varepsilon\bE_x\left[ 1_{\{X_{\tau} \leq R\}} Z^n(t,\tau) e^{A^n_m(t+\tau_m,X_{\tau_m})+C^n(t+\tau_m,t+\tau,X_{\tau_m},X_{\tau}))}\right]+\varepsilon K_1,
\end{multline}
where $K_1:= e^{3\Vert c\Vert+ 6T\Vert f\Vert}$. Let $D:=\{\sup\limits_{s\in [0,\delta_m]}\rho(X_s,X_0)\leq r\}$. Using~\eqref{eq:acalc4}--\eqref{eq:acalc6}, on the set $\{X_{\tau} \leq R\}$, we get
\begin{align*}
\bE_{X_{\tau}}\left[ 1_{D} \,\exp\left(\sup\limits_{s\in [0,T]}\sup\limits_{v\in [0,\delta_m]}A^n_m(s,X_v)+\sup\limits_{|s-u|\leq \delta_m}\sup\limits_{v \in [0,\delta_m] }C^n(s,u,X_v,X_0)\right)\right] & \leq e^{ 2\varepsilon},\\
\bE_{X_{\tau}}\left[ 1_{D'}\, \exp\left(\sup\limits_{s\in [0,T]}\sup\limits_{v\in [0,\delta_m]}A^n_m(s,X_v)+\sup\limits_{|s-u|\leq \delta_m}\sup\limits_{v \in [0,\delta_m] }C^n(s,u,X_v,X_0)\right)\right] & \leq \varepsilon K_2,
\end{align*}
where $K_2:= e^{2\Vert c\Vert + 6 T\Vert f\Vert}$. Thus, using strong Markov property, for any $t\in [0,T]$ and $x\in \Gamma$, we get
\begin{equation*}
\bE_x\left[ 1_{\{X_{\tau} \leq R\}} Z^n(t,\tau) e^{A^n_m(t+\tau_m,X_{\tau_m})+C^n(t+\tau_m,t+\tau,X_{\tau_m},X_{\tau}))}\right]\leq e^{ 2\varepsilon}\bE_x\left[Z^n(t,\tau)\right]+ \varepsilon K_3,
\end{equation*}
where $K_3:= K_2 e^{\Vert c\Vert+2 T\Vert f\Vert}$. Hence, recalling~\eqref{eq:acalc7}, we get
\begin{equation}\label{eq:acalc8}
\bE_x\left[Z^n(t,\tau) e^{B^n_m(t+\tau_m,t+\tau,X_{\tau_m},X_{\tau}))}\right]\leq e^{3\varepsilon} \bE_x\left[Z^n(t,\tau)\right]+\varepsilon e^{\varepsilon} K_3+\varepsilon K_1.
\end{equation}
Recalling that $\Vert f\Vert \delta_m\leq \varepsilon$ and combining~\eqref{eq:acalc8} with~\eqref{eq:acalc1}, for any $t\in [0,T]$ and $x\in \Gamma$, we finally get
\begin{align*}
0\leq h^{n+1}_m(t,x)-h^{n+1}(t,x) & \leq  \left(e^{4\varepsilon}-1\right)e^{2T\Vert f\Vert+\Vert c\Vert} + \varepsilon e^{2\varepsilon} K_3+ \varepsilon e^\varepsilon K_1+\varepsilon.
\end{align*}
Noting that the upper bound is independent of $t$ and uniform on $\Gamma$, and recalling that $\varepsilon>0$ was arbitrary, we get $h^{n+1}_m(t,x)\to h^{n+1}(t,x)$ as $m\to \infty$ uniformly on $[0,T]\times \Gamma$. Thus, recalling~\eqref{eq:acalc0}, we conclude the proof.
\end{proof}
Finally, we show that the optimal value of the finite horizon impulse control problem may be approximated by the strategies with finitely many impulses.
\begin{lemma}\label{lm:finite_horizon_conv2}
For any $T\in \bN$ and $x\in E$ we get $\inf\limits_{V\in \bV^n} J_T(x,V) \to \inf\limits_{V\in \bV} J_T(x,V)$ as $n\to \infty$.
\end{lemma}

\begin{proof}
Using monotonicity of the exponent function, it is enough to show
\begin{equation}\label{eq:ccccl1}
\lim_{n\to \infty}\inf_{V\in \bV^n} \exp\left(J_T(x,V)\right)= \inf_{V\in \bV}\exp\left( J_T(x,V)\right).
\end{equation}

Let $\varepsilon>0$ and $V^\varepsilon=(\tau_i,\xi_i)_{i=1}^{\infty}\in \bV$ be an $\varepsilon$-optimal strategy for $\inf_{V\in \bV}e^{J_T(x,V)}$. For any $n\in \bN$, let $V^\varepsilon_n\in \bV^n$ denote the restriction of $V^\varepsilon$ to the first $n$ impulses. Then, for any $n\in \bN$, we get
\begin{equation}\label{eq:ccccl1.5}
0 \leq \inf_{V\in \bV^n}e^{J_T(x,V)}-\inf_{V\in \bV} e^{J_T(x,V)}\leq  e^{J_T(x,V^\varepsilon_n)}-e^{J_T(x,V^\varepsilon)}+\varepsilon.
\end{equation}
For clarity, we use $X$ and $Y$ to denote the processes controlled by the strategies $V^\vep$ and $V^\vep_n$, respectively.
Furthermore, we set $N_T:=\sum_{i=1}^\infty 1_{\{\tau_i\leq T\}}$ and
\begin{align*}
A_n(x):= & \bE_{(x,V^\vep_n)}\left[1_{\{N_T> n\}} \exp\left(\int_0^{T} f(Y_s) ds+\sum_{i=1}^n 1_{\{\tau_i\leq T\}}c(Y_{\tau_i^-},\xi_i)\right)\right]\nonumber\\
& - \bE_{(x,V^\vep)}\left[1_{\{N_T> n\}} \exp\left(\int_0^{T} f(X_s) ds+\sum_{i=1}^\infty 1_{\{\tau_i\leq T\}}c(X_{\tau_i^-},\xi_i)\right)\right],
\end{align*}
for any $n\in\bN$ and $x\in E$. Noting that, for any $n\in \bN$, we get
\begin{multline*}
\bE_{(x,V^\vep_n)}\left[1_{\{N_T\leq n\}} \exp\left(\int_0^{T} f(Y_s) ds+\sum_{i=1}^n 1_{\{\tau_i\leq T\}}c(Y_{\tau_i^-},\xi_i)\right)\right] \\
=\bE_{(x,V^\vep)}\left[1_{\{N_T\leq n\}} \exp\left(\int_0^{T} f(X_s) ds+\sum_{i=1}^\infty 1_{\{\tau_i\leq T\}}c(X_{\tau_i^-},\xi_i)\right)\right],
\end{multline*}
and recalling \eqref{eq:ccccl1.5}, we get
\[
0 \leq \inf_{V\in \bV^n}e^{J_T(x,V)}-\inf_{V\in \bV} e^{J_T(x,V)}\leq |A_n(x)|+\epsilon.
\]
Next, using boundedness of $f$, non-negativity of $c$ and the fact that $V^\vep_n$ is the restriction of $V^\vep$, we get
\begin{equation*}
|A_n(x)| \leq 2\bE_{(x,V^\vep)}\left[1_{\{N_T> n\}} \exp\left(T \Vert f \Vert+\sum_{i=1}^\infty 1_{\{\tau_i\leq T\}}c(X_{\tau_i^-},\xi_i)\right)\right].
\end{equation*}
Consequently, noting that $\bE_{(x,V^\vep)}\left[e^{\sum_{i=1}^\infty 1_{\{\tau_i\leq T\}}c(X_{\tau_i^-},\xi_i)}\right]<\infty$ and $1_{\{N_T> n\}}\to 0$ as $n\to \infty$, and using bounded convergence theorem, we get $|A_n(x)|  \to 0$, as $n\to\infty$, which concludes the proof.
\end{proof}

\end{document}